\documentclass[12pt,preprint]{elsarticle}

\usepackage{xspace,exscale}
\usepackage{fancybox}
\usepackage{graphicx}
\usepackage{color}
\usepackage{amsfonts}
\usepackage{amssymb}
\usepackage{url}

\usepackage{amsmath}
	\makeatletter
	\let\over=\@@over \let\overwithdelims=\@@overwithdelims
	\let\atop=\@@atop \let\atopwithdelims=\@@atopwithdelims
  	\let\above=\@@above \let\abovewithdelims=\@@abovewithdelims
  	\makeatother
\usepackage{fixltx2e}

\usepackage{rotating}

\usepackage{ifpdf}

\usepackage{subfigure}
\usepackage{psfrag}

\usepackage[bbgreekl]{mathbbol}

\newcommand{\mreals}{\ensuremath{\mathbb{R}}}

\newcommand{\mproj}{\ensuremath{\mathrm{proj}}}
\newcommand{\FF}{\ensuremath{\mathbb{F}}}
\def\krawt{K^{(n)}}
\def\krawtk{K^{(k)}}
\newcommand{\PP}{\ensuremath{\mathbb{P}}}
\newcommand{\Sph}{\ensuremath{\mathbb{S}}}

\ifx\eqref\undefined
	\newcommand{\eqref}[1]{~(\ref{#1})}
\fi
\ifx\mod\undefined
	\def\mod{\mathop{\rm mod}}
\fi

\def\tr{\mathop{\rm tr}}

\def\EE{\mathbb{E}\,}

\def\PP{\mathbb{P}}

\def\diag{\mathop{\rm diag}}

\def\follows{\Longrightarrow}

\def\eqdef{\stackrel{\triangle}{=}}

\def\upto{\nearrow}
\def\downto{\searrow}

\def\unifto{\mathop{{\mskip 3mu plus 2mu minus 1mu%
	\setbox0=\hbox{$\mathchar"3221$}%
	\raise.6ex\copy0\kern-\wd0%
	\lower0.5ex\hbox{$\mathchar"3221$}}\mskip 3mu plus 2mu minus 1mu}}

\ifx\lesssim\undefined
\def\simleq{{{\mskip 3mu plus 2mu minus 1mu%
	\setbox0=\hbox{$\mathchar"013C$}%
	\raise.2ex\copy0\kern-\wd0%
	\lower0.9ex\hbox{$\mathchar"0218$}}\mskip 3mu plus 2mu minus 1mu}}
\else
\def\simleq{\lesssim}
\fi

\ifx\gtrsim\undefined
\def\simgeq{{{\mskip 3mu plus 2mu minus 1mu%
	\setbox0=\hbox{$\mathchar"013E$}%
	\raise.2ex\copy0\kern-\wd0%
	\lower0.9ex\hbox{$\mathchar"0218$}}\mskip 3mu plus 2mu minus 1mu}}
\else
\def\simgeq{\gtrsim}
\fi

\newcommand{\alert}[1]{\textcolor{red}{#1}}

\newif\ifmapx
{\catcode`/=0 \catcode`\\=12/gdef/mkillslash\#1{#1}}
\edef\jobnametmp{\expandafter\string\csname abmaps_apx\endcsname}
\edef\jobnameapx{\expandafter\mkillslash\jobnametmp}
\edef\jobnameexpand{\jobname}
\ifx\jobnameexpand\jobnameapx
\mapxtrue
\else
\mapxfalse
\fi

\long\def\apxonly#1{\ifmapx{\color{blue}#1}\fi}

\newtheorem{theorem}{Theorem}
\newtheorem{lemma}[theorem]{Lemma}
\newtheorem{corollary}[theorem]{Corollary}
\newtheorem{definition}{Definition}
\newtheorem{proposition}[theorem]{Proposition}
\newtheorem{remark}{Remark}
\newproof{proof}{Proof}

\begin{document}

\title{On metric properties of maps between Hamming spaces and related graph homomorphisms}

\author[yp]{Yury Polyanskiy}
\ead{yp@mit.edu}
\address[yp]{Department of Electrical Engineering 
and Computer Science, MIT, Cambridge, MA 02139 USA.}

\begin{abstract} 
A mapping of $k$-bit strings into $n$-bit strings is called an $(\alpha,\beta)$-map if $k$-bit strings which are more
than $\alpha k$ apart are mapped to $n$-bit strings that are more than $\beta n$ apart. This is a relaxation of the
classical problem of constructing error-correcting codes, which corresponds to $\alpha=0$. Existence of an $(\alpha,\beta)$-map is equivalent to 
existence of a graph homomorphism $\bar H(k,\alpha k)\to \bar H(n,\beta n)$, where $H(n,d)$ is a Hamming graph with
vertex set $\{0,1\}^n$ and edges connecting vertices differing in $d$ or fewer entries.

This paper proves impossibility results on achievable parameters $(\alpha,\beta)$ in the regime of $n,k\to\infty$ with a
fixed ratio ${n\over k}= \rho$. This is done by developing a general criterion for existence of graph-homomorphism based on the semi-definite relaxation of the independence
number of a graph (known as the Schrijver's $\theta$-function). The criterion is then evaluated using some known and
some new results from coding theory concerning the $\theta$-function of Hamming graphs.
As an example, it is shown that if $\beta>1/2$ and $n\over k$ -- integer, the ${n\over k}$-fold repetition map achieving
$\alpha=\beta$ is asymptotically optimal.
%

Finally, constraints on configurations of
points and hyperplanes in projective spaces over $\FF_2$ are derived.
\end{abstract}

\begin{keyword}
Error-correcting codes \sep graph homomorphism \sep Schrijver's $\theta$-function \sep projective geometry over $\FF_2$
\end{keyword}

\maketitle


\section{Introduction}

Hamming space $\FF_2^k$ of binary $k$-strings, equipped with the Hamming distance is one of the classical
objects studied in combinatorics. Its properties that received significant attention are the maximal packing densities, covering numbers,
isoperimetric inequalities, list-decoding properties, etc. In this paper we are interested in studying metric properties of maps 
$f:\FF_2^k\to\FF_2^n$ between Hamming spaces of different dimensions. 

Indeed, frequently one is interested in embedding $\FF_2^k$ into $\FF_2^n$ ``expansively'', i.e. 
so that points that were far apart in
$\FF_2^k$ remain far apart in $\FF_2^n$. Two immediate examples of such maps are: the 
error-correcting codes with rate $k/n$ and minimum distance $d$ satisfy
$$ |x-x'|>0 \implies |f(x) - f(x')| \ge d \,,$$
where here and below $|z| = \|z\|_0= |\{i: z_i \neq 0\}|$ is the Hamming weight of the vector. Another example is the  
 repetition coding with $f(x)$ mapping $x$ into $n\over
k$ repetitions of $x$. This map satisfies:
\begin{equation}\label{eq:rep}
	|x-x'|>\alpha k \implies |f(x)-f(x')|>\alpha n\,.
\end{equation}

With these two examples in mind, we introduce the main concept of this paper.
\begin{definition}\label{def:ab} A map $f:\FF_2^k \to \FF_2^n$ is called an $(\alpha, \beta; k,n )$-map (or simply an
$(\alpha,\beta)$-map) if $\alpha k$ and $\beta n$ are integers and for all $x,x'\in \FF_2^k$ we have either
\begin{equation}\label{eq:def_ab}
	|f(x) - f(x')| > \beta n \quad \mbox{or} \quad |x-x'|\le\alpha k\,, 
\end{equation}
where $\FF_2^k$ is the Hamming space of dimension $k$ over the binary field.
\end{definition}

We next define the Hamming graphs $H(n,d)$ for integer $d\in[0,n]$ as follows:
\begin{equation}\label{eq:hnd_def}
	V(H(n,d)) = \FF_2^n, \quad E(H(n,d)) = \{(x,x'): 0<|x-x'| \le d\}\,.
\end{equation}
\apxonly{The Johnson graphs $J(n,d,w)$ are defined for integers $d \le 2w, w\in[0, n/2]$ as an induced subgraph of $H(n,d)$ with
vertices
$$ V(J(n,d,w)) = \{x\in\FF_2^n: |x| = w\}\,. $$}
By $V(G), E(G)$ and $\bbalpha(G)$ we denote the vertices of $G$, the edges of $G$ and the cardinality of the maximal
independent set of $G$. All graphs in this paper are simple (without self-loops and multiple edges). 
By $\bar G$ we denote the (simple) graph
obtained by complementing $E(G)$ and deleting self-loops.

The relevance of Hamming graphs to this paper comes from the simple observation: 
$$ \exists \mbox{$(\alpha, \beta; k,n)$-map} \quad \iff \quad \bar H(k, \alpha k) \to \bar H(n, \beta n)\,, $$
where $G\to H$ denotes the existence of a graph homomorphism (see Section~\ref{sec:hom} for definition).

This paper focuses on proving negative results showing impossibility of certain parameters $(\alpha, \beta)$.
Note that there are a variety of methods that we can use to disprove existence of graph homomorphisms. For example, by
computing the shortest odd cycle we can prove
$$  \bar H(2,0) \not\to \bar H(4,2) \not\to \bar H(6,4) \not\to \bar H(8,6) \not\to \cdots\,. $$
In this paper, however,  we are interested in the methods that provide some useful information in the asymptotic
regime of $k\to \infty$, ${n\over k} \to \rho > 0$ and fixed $(\alpha,\beta)$.

\subsection{More on the concept of an $(\alpha,\beta)$-map}

Our original motivation for Definition~\ref{def:ab} was the following. Suppose the map $f:\FF_2^k\to\FF_2^n$ is used to 
protect the $k$ data bits against noise. 
If the points $x,x' \in \FF_2^k$ are far apart but $f(x)$ and $f(x')$ are close, i.e. if a map fails to 
satisfy~\eqref{eq:def_ab}, then $f(x)$ may be confused with $f(x')$ in a noisy environment. Consequently, this would lead
to a severe discrepancy if $x'$ is reported instead of $x$.

Below we briefly discuss how $(\alpha,\beta)$-property relates to some previously studied concepts.

First, a $(0,\beta)$-map is simply an error-correcting code of rate $k/n$ and minimum distance $1+\beta n$. Thus,
$(\alpha,\beta)$-condition is a \textit{relaxation of the minimum-distance property:} the separation of $1+\beta n$
is only guaranteed for data vectors $x,x'$ that were $1+\alpha k$ apart to start with. Practically, data may have some
structure guaranteeing some separation between feasible data-vectors (e.g. if $x$ is English test, changing one letter
is unlikely to result in a grammatically correct phrase). 

Second, in the inverse problem of reconstructing $x$ from a noisy version $y=f(x)+z$, one may proceed by computing a
pre-image of the Hamming ball of radius $\beta n/2$ around $y$. Then the $(\alpha,\beta)$-condition guarantees that the
points in the pre-image will all be close to each other. 
\apxonly{Note: People commented that for error-correction it may in fact be better if these ``lists'' had large
separation, since then it would be easier to weed off errors. This corresponds to the problem:
$$ 0<|x-x'| \le \gamma k \implies |f(x) - f(x')| > \delta n\,. $$
The bounds I got on this are:
$$ R \le R_{LP}(\delta)/h(\gamma/2)\\ R \le (1-h(\delta/2))/(1-R_{LP}(\gamma)).$$
}

Third, an $(\alpha,\beta)$-map can be used to convert a code with (normalized) minimal distance $> \alpha$ to a code of
minimal distance $> \beta$ at the expense of losing a factor ${k/n}$ in rate. This observation leads, on one hand, to a
non-trivial bound on achievable parameters $(\alpha,\beta)$, see~\eqref{eq:ccb} below. On the other hand, it also
suggests that $(\alpha,\beta)$-maps could be employed for adapting properties of a fixed mother code to the changing
noise environment.

Fourth, an $(\alpha,\beta)$-map with $n<k$ can be seen as a type of hashing in which one wants the hashes
of dissimilar strings to be also dissimilar\apxonly{\footnote{Note that the opposite requirement: $|x-x'|\le \alpha k\implies
|f(x)-f(x')|\le \beta n$ is trivially fulfilled by the constant map. Thus to avoid triviality in the study of this
variation it is necessary to put an extra constraint that $f$ be injective. This variation, however, does not reduce to the
question of existence of graph-homomorphisms and can not be treated with methods of this paper.}}. In fact, the
$(\alpha,\beta)$-condition is weakening of the locality-sensitive
hashing condition~\cite{IM98,AP06}.
\apxonly{Here is more. If $h:\FF_2^k\to\{0,1\}$ is a $(R, cR, P_1, P_2)$ LSH for $\FF_2^k$ then (in particular)
$$ \forall |x-x'|>cR: \PP[h(x) \neq h(x')] \ge 1-P_2\,. $$
Thus by producing an $n$-vector $(h_1,\ldots,h_n)$ we get with high probability
$$ |x-x'|>cR \quad \implies \quad |h^n(x) - h^n(x')| \ge n(1-P_2)\,.$$
}%

Finally, relaxation of the minimum-distance property taken in Definition~\ref{def:ab} may be motivated by availability
of the redundancy in the $k$-bit data. In information theory transmitting such data across a noisy channel is known as the joint
source-channel coding (JSCC) problem. Combinatorial variation, cf.~\cite{KMP12isit,YP15-cjscc_isit}, can be stated as
follows: say that $f:\FF_2^k \to \FF_2^n$ is a $(D,\delta)$-JSCC if there exists
a decoder map $g:\FF_2^n\to\FF_2^k$ with the property 
$$ \forall x\in \FF_2^k, z\in\FF_2^n: \quad |f(x) - z| \le \delta n \quad\implies\quad |x - g(z)| \le Dk\,.$$
The operational meaning is that a $(D,\delta)$-JSCC reduces the (adversarial) noise of strength $\delta$ in $n$-space to 
(adversarial) noise of strength $D$ in $k$-space. A special case of $D=\epsilon \delta$ was introduced by
Spielman~\cite{DS96} under the name of \textit{error-reducing} codes. The connection to Def.~\ref{def:ab} comes from the
simple observation:
$$ f\mbox{~is a $(D,\delta)$-JSCC} \quad\implies\quad f\mbox{~is a $(2D,2\delta)$-map}\,.$$
Thus, every impossibility result for $(\alpha,\beta)$-maps implies impossibility results for $(D,\delta)$-JSCC and
Spielman's error-reducing codes.

\section{Main results}

For $\alpha=0$ the best known bound to date is due to McEliece et al~\cite{MRRW77}. It says that any set $S \subset
\FF_2^n$ with $|y-y'|\ge \delta n$ for all $y,y'\in S, y\neq y'$ satisfies
\begin{equation}\label{eq:mrrw}
	{1\over n} \log |S| \le R_{LP2}(\delta) + o(1)\,,
\end{equation}
where $R_{LP2}(\delta) = 0$ for $\delta \ge 1/2$ and for $\delta < 1/2$:
	\begin{align} 	R_{LP2}(\delta) &= \min (1 -h(\alpha) + h(\beta))\,, \label{eq:rlp2}
	\end{align}		
	where $h(x)=-x \log x - (1-x) \log(1-x)$ and the minimum is taken over all $0\le\beta\le\alpha\le1/2$ satisfying
		$$ 2{\alpha (1-\alpha) - \beta(1-\beta)\over 1+2\sqrt{\beta (1-\beta)}} \le \delta\,. $$
	For distances $\delta < 0.273$ the solution is given by $\alpha=1/2$ and $R_{LP2}(\delta)$ has a simpler
	expression: 
	\begin{align}\label{eq:rlp1}
		R_{LP1}(\delta) = h(1/2 - \sqrt{\delta(1-\delta)})\,.
\end{align}	
	
Thus from~\eqref{eq:mrrw} we get
	$$ \exists \mbox{$(0,\beta;\, k,n)$-map} \quad \implies {k/n} \simleq R_{LP2}(\beta)\,.$$
A natural question is whether going from $\alpha=0$ to $\alpha>0$ may enable larger rates $k/n >
R_{LP2}(\beta)$. 

The first impulse could be that the answer is negative. Indeed, note that for $\alpha<1/2$ there is $2^{k+o(k)}$ points $x'$ s.t. $|x-x'|>\alpha k$. 
Thus it may seem that for $\alpha<1/2$ this relaxation yields 
no improvements (asymptotically) compared to $\alpha=0$. This observation is incorrect for two reasons. 
First, 
we do not require $f$ to be injective --- thus although all points $f(x')$ are far from $f(x)$, they may not all be
distinct. Second, even though each $x$ has many $x'$ satisfying $|x-x'|>\alpha k$, we in
fact need a collection $S\subset \FF_2^k$ s.t. $|x-x'|>\alpha k$ for all \textit{pairs} $x,x'\in S$. Only then we may
conclude that $f(S)$ is code in $\FF_2^n$ with large minimal distance.

Thus, $S$ needs to be an independent set in $H(k, \alpha k)$. How large can it be? To that end, we recall Turan's
theorem, cf.~\cite[Theorem IV.6]{bollobas2013modern}:
	\begin{equation}\label{eq:turan}
		\bbalpha(G) \ge {|V(G)|^2\over 2 (|E(G)| + |V(G)|)}\,.
	\end{equation}
Counting the number of edges of the graph $H(k, \alpha k)$ via Stirling's formula we get $|E(H(k,\alpha k))|= 2^k
\sum_{j=0}^{\alpha k} {k\choose j} = 2^{k+k h(\alpha) + o(k)}$. Therefore,
\begin{equation}\label{eq:gv1}
	\bbalpha(H(k, \alpha k)) \ge {(2^k)^2 / 2\over |E(H(k, \alpha k))| + 2^k} = 2^{k (1-h(\alpha)) + o(k)}\,.
\end{equation}
Consequently, if an $(\alpha, \beta)$-map exists then comparing~\eqref{eq:mrrw} and~\eqref{eq:gv1} we get
\begin{equation}\label{eq:ccb}
	k (1-h(\alpha)) + o(k) \le n R_{LP2}(\beta) + o(n)\,. 
\end{equation}

One natural way to improve the bound would be to notice that graphs $H(k,
\alpha k)$ have a lot of extra structure and perhaps simplistic estimate~\eqref{eq:gv1} via Turan's theorem can be improved. Unfortunately, despite decades of
work the lower bound~\eqref{eq:gv1}, known as the Gilbert-Varshamov bound, is asymptotically the best known. (For
non-binary alphabets, however, better bounds exist~\cite{tsfasman1982modular}.)

Instead, the next theorem improves~\eqref{eq:ccb} by establishing how another graph-function (the
$\theta$-function, see~\eqref{eq:ts_max} below) behaves under graph homomorphisms, and then applying known results on
$\theta$-function for Hamming graphs established by Samorodnitsky~\cite{AS01,navon2005delsarte} and McEliece et
al.~\cite{MRRW77}.
\begin{theorem}\label{th:main1} For every $\epsilon>0$ there exist a sequence $\delta_m\to 0$ s.t. if an $(\alpha,
\beta; k, n)$-map exists with $\alpha \ge \epsilon$ and $\beta\ge \epsilon$ then
	\begin{align} 
	   k R_{Sam}(\alpha) + k \delta_k &\le n R_{LP2}(\beta) + n \delta_n \qquad \label{eq:ccsam} \\
	   	\intertext{and}
	   k \left(1-h\left({\alpha\over2}\right)\right) + k \delta_k &\le 
			n \left(1-h\left({\beta\over2}\right)\right) + n \delta_n\,,  \label{eq:it} 
	\end{align}
	where 
		\begin{equation}\label{eq:rsam_def}
			R_{Sam}(\alpha) = {1\over2}\max\left(1-h(\alpha) + R_{LP1}(\alpha), 
		h(1-2\sqrt{\alpha(1-\alpha)}) \right) 
\end{equation}		
	for $\alpha<1/2$ and zero otherwise.
\end{theorem}

\begin{remark} The bound~\eqref{eq:ccsam} is better for $n/k > 1$, while~\eqref{eq:it} is better for $n/k < 1$. See
Section~\ref{sec:disc} for evaluations.
\end{remark}

Note that by virtue of relying only on the number of edges in $\bar H(k, \alpha k)$ the bound in~\eqref{eq:ccb} is
robust in the sense that whenever $(\alpha, \beta)$ violate~\eqref{eq:ccb}, there will be great many pairs of 
$x,x'$ that violate~\eqref{eq:def_ab}. Here is a similar strengthening of Theorem~\ref{th:main1}.

\begin{theorem}\label{th:main2} For every $\epsilon>0$ there exist a sequence $\delta_m\to 0$ with the following
property. For every map $f:\FF_2^k \to \FF_2^n$, every $S\subset \FF_2^k$ of size $|S|>2^{k(1-\epsilon+\delta_k)+n\delta_n}$
and every $\alpha,\beta \in[\epsilon, 1]$ satisfying
\begin{align} 
    k R_{Sam}(\alpha) - k\epsilon &\ge n R_{LP2}(\beta)\label{eq:tm2_1}\\
	\intertext{or}
    k \left(1-h\left({\alpha\over2}\right)\right) - k\epsilon &\ge n \left(1-h\left({\beta\over2}\right)\right) 
    \label{eq:tm2_2}
\end{align}
 there exists a pair $x,x'\in S$ such
that
\begin{equation}\label{eq:tm2_2a}
	|x-x'| > \alpha k \quad\mbox{and}\quad |f(x)-f(x')| \le \beta n\,.
\end{equation}
In particular, there are at least $2^{k(1+\epsilon - \delta_k) - n\delta_n}$ un-ordered pairs $\{x, x'\} \subset \FF_2^k$
satisfying~\eqref{eq:tm2_2a}.
\end{theorem}

\medskip

Next we consider an improved bound for the case of $\beta>1/2$. Notice that by Plotkin bound~\cite[Chapter 2.2]{MS1977} we have
$$ \bbalpha(H(n, \beta n)) \le 1 + {n \over 2\beta n + 2 -n}\,.$$
In particular, $\bar H(n, \beta n)$ does not contain $K_4$ whenever $\beta > {2/3}$. Therefore, any graph $G$ which
contains $K_4$ cannot map into $\bar H(n, \beta n)$. For example:
$$ \bar H(3, 1) \not \to \bar H(n, \beta n) \qquad \forall n\in\mathbb{Z}_+, \beta > 2/3\,.$$

The following elaborates on this idea:
\begin{theorem}\label{th:main3} For every $\epsilon > 0$ there exists $\delta_m\to0$ such that if there exists an
$(\alpha, \beta; k, n)$ map with $\beta>{1\over2}$ and $\alpha\in[{1\over 2} +\epsilon; 1-\epsilon]$ then
	\begin{equation}\label{eq:tm3}
		\alpha \ge \beta + {(2\beta -1)^2\over 2} \delta_k\,.
\end{equation}	
	Furthermore, for any map $f:\FF_2^k\to \FF_2^n$, any $\beta > 1/2$ and $\alpha\in[{1\over 2} +\epsilon; 1-\epsilon]$ and 
any set $S\subset \FF_2^k$ of size
	$$ |S| > 2^k {{2\beta \over 2\beta -1}\over {2\alpha\over 2\alpha -1} - \delta_k} $$
	there exists a pair of points $x,x'\in S$ satisfying~\eqref{eq:tm2_2a}.
\end{theorem}
\begin{remark} Considering the argument preceding the theorem, it should not be so surprising that the relation between
$\alpha$ and $\beta$ in~\eqref{eq:tm3} is independent of the rate $k\over n$. 
The significance of~\eqref{eq:tm3} is that for the case of ${n \over k} \in \mathbb{Z}$ this bound is
(asymptotically) optimal, as the
example of the repetition map~\eqref{eq:rep} clearly shows. \apxonly{This corresponds to the classical fact that parameters of
the best codes for minimal distance $d  > n/2$ are known exactly thanks to results of Plotkin and Levenshtein,
cf~\cite[Chapter 2, Theorem 8]{MS1977}.} For linear $(\alpha,\beta)$-maps the result was shown in~\cite[Theorem
8]{YP15-cjscc_isit} by studying properties of the generator matrix.
\end{remark}

When applied to linear maps $\FF_2^k\to\FF_2^n$ Theorems~\ref{th:main2} and~\ref{th:main3} have the following geometric
interpretations:
\begin{corollary}\label{th:cor2} For every $\epsilon>0$ there exists a sequence $\delta_\ell\to 0$ with the following
property. Fix any two
lists of (possibly repeated) points $u_1, \ldots, u_k$ and $v_1,\ldots, v_n$ in projective space $\PP^{m-1}(\FF_2)$ s.t.
that they are not all contained in a codimension 1 hyperplane. Fix any $\alpha,\beta \in [\epsilon,1]$ s.t.
	\begin{equation}\label{eq:corA}
		m > k (1-R_{Sam}(\alpha)+\delta_k) + n (R_{LP2}(\beta) +\delta_n) 
	\end{equation}	
	or
	\begin{equation}\label{eq:corB} m > k (h\left(\alpha/2\right) +\delta_k) +n (1-h(\beta/2)+\delta_n)\,.
	\end{equation}	
	There exists a hyperplane $H$ of codimension 1 in $\PP^{m-1}$ such that
	\begin{equation}\label{eq:ab_hlin}
	\#\{j: v_j\in H\} \ge (1-\beta) n\,, \quad\#\{i: u_i \in H\} < (1-\alpha) k \mbox{~or~}=k\,.
	\end{equation}	
\end{corollary}

\begin{corollary} \label{th:cor3} For every $\epsilon > 0$ there exists $\delta_\ell \to 0$ with the following property. Fix any two
lists of (possibly repeated) points $u_1, \ldots, u_k$ and $v_1,\ldots, v_n$ in projective space $\PP^{m-1}(\FF_2)$ s.t.
that they are not all contained in a codimension 1 hyperplane. Fix any $\beta > 1/2$ and $\alpha\in[{1\over 2}
+\epsilon; 1-\epsilon]$ s.t.
\begin{equation}\label{eq:corC}
	m > k + \log_2 {2\beta\over 2\beta -1} - \log_2\left({2\alpha\over 2\alpha -1}-\delta_k\right)\,. 
\end{equation}
Then there exists a hyperplane $H$ of codimension 1 in $\PP^{m-1}$ satisfying~\eqref{eq:ab_hlin}.
\end{corollary}

Note that by identifying homogeneous coordinates with affine coordinates we can establish set-isomorphism
$\PP^{m-1}(\FF_2)$ and $\FF_2^m\setminus \{0\}$. Thus, previous corollaries can be equivalently restated in terms of 
$\FF_2^m$. For example, for any $\epsilon>0$, all $k$ sufficiently large and all $n$: \textit{Fix some basis of
$\FF_2^k$ and arbitrary non-zero points $v_1,\ldots,v_n \in \FF_2^k$. Then there exists a $(k-1)$-subspace containing
$\ge {n\over 4}$ $v$-points and $<\left({1\over4}+\epsilon\right)k$ basis vectors.} Note that this is a manifestly
$\FF_2$-property since over large fields one could select $v$-points (when $n>4k$) so that no ${n\over 4}$ of them are
contained in a $(k-1)$-subspace. \apxonly{E.g. Vandermonde matrix.}

\medskip
The rest of the paper is organized as follows: Section~\ref{sec:hom} proves a few results on graph homomorphisms. In
Section~\ref{sec:hamming} these results are applied to prove Theorems~\ref{th:main1}-\ref{th:main3} and
Corollaries~\ref{th:cor2}-\ref{th:cor3}. We conclude in
Section~\ref{sec:disc} with discussion, numerical evaluations and some open problems.

\section{Graph homomorphisms}\label{sec:hom}

Let us introduce notation to be used in the remainder of the paper:
\begin{align} 
   \theta_S(G) &\eqdef \max \{\tr JM: \tr M=1, M\succeq 0, M|_{E(G)} = 0, M_{v,v'} \ge 0\,\, \forall v,v'\}
   			\label{eq:ts_max}\\
   		&= \min\{\lambda_{\max}(C): C=C^T, C|_{E(G)^c} \ge 1\} \label{eq:ts_min}\\
   \theta_L(G) &\eqdef \max \{\tr JM: \tr M=1, M\succeq 0, M|_{E(G)} = 0\}\\
   		&= \min \{\lambda_{\max}(C): C=C^T, C|_{E(G)^c}=1\}\,,\label{eq:tl_min}
\end{align}   
where $M$ is a positive-semidefinite matrix of order $|V(G)|$, $J$ is an all-one matrix of the same size,
$\lambda_{\max}(\cdot)$ denotes the maximal eigenvalue and $M|_{S}$ denotes a subset $\{M_{i,j}: (i,j)\in S\}$ of the
entries of matrix $M$, so that $M|_{E(G)} = \{M_{i,j}: i\sim j\mbox{~ in~} G\}$.   $\theta_S(G)$ and
$\theta_L(G)$ are the 
Schrijver and Lov\'asz $\theta$-functions, respectively\footnote{Other authors write $\theta(G)$ for $\theta_L(G)$ 
and any of $\theta'(G)$, $\theta_{1/2}(G)$ or $\theta^-(G)$ for $\theta_S(G)$.}. 

We recall a few properties of the $\theta$-function (one may
consult~\cite{goemans1997semidefinite} for more): 
\begin{itemize}
\item Both $\theta$-functions are typically used to upper bound the independence number
of a graph:
\begin{equation}\label{eq:lova_0}
	\bbalpha(G) \le \theta_S(G) \le \theta_L(G)\,.
\end{equation}
\item $\theta_L(\cdot)$, while yielding a looser bound on $\bbalpha(\cdot)$, is multiplicative under strong product of
graphs\footnote{The strong product $G \boxtimes H$ is a simple graph with vertex set given by $V(G)\times V(H)$ and edges $(g_1,h_1)\sim (g_2,h_2): (g_1\sim
	g_2 \mbox{~or~} g_1=g_2)\mbox{~and~}(h_1\sim h_2 \mbox{~or~} h_1=h_2)$.} as shown in~\cite{LL79}:
 \begin{equation}\label{eq:lova_1}
 	\theta_L(G \boxtimes H) = \theta_L(G)\theta_L(H)\,. 
\end{equation} 
\item For a vertex transitive graphs, we also have reciprocity~\cite{LL79}:
\begin{equation}\label{eq:lova_2}
			\theta_L(G) \theta_L(\bar G) = |V(G)|\,.
\end{equation}		
\end{itemize}
Our main technical contribution in this section is the following partial generalization
of~\eqref{eq:lova_1}-\eqref{eq:lova_2} to $\theta_S$:
\begin{lemma}\label{th:sprod} Let $G$ be vertex transitive, then
\begin{align}
	\theta_S(G\boxtimes H) &\le {|V(G)|\over \theta_S(\bar G)} \theta_S(H)\,.\label{eq:sprod}
\end{align}
	\apxonly{Also:
	\begin{equation}
 	\bbalpha(G\boxtimes H) \le {|V(G)|\over \bbalpha(\bar G)} \bbalpha(H)\,.\label{eq:sproda}
	\end{equation}
	}
\end{lemma} 
Proof is given at the end of this section. We next discuss its application to existence of graph homomorphisms.

\apxonly{Some identities:
\begin{itemize}
\item Alternative dual versions:
\begin{align} \theta_S(G) &=  \min\{\max_v D_{v,v}: (D-J)\succeq 0, D|_{E(\bar G)}\le 0\}\label{eq:thetas_alt}\\
	      \theta_L(G) &=  \min\{\max_v D_{v,v}: (D-J)\succeq 0, D|_{E(\bar G)}= 0\}
\end{align}
	Passing between $C$ and $B$ is done as follows:
	$$ C=(\max_v D_{v,v})I - D+J \qquad D=J+\lambda_{\max}(C) I-C\,.$$
	It is in this form that $\theta_L(G_1\boxtimes G_2)$ has solution $D_1\otimes D_2$.
\item Another version is original ``umbrella'' formulation of Lov\'asz. Given $G$ let $H:\mreals^{|G|}\to\mreals^d$ be a
map such that: a) $\|H\delta_v\|=1$ all $v$; b) $v\not\sim v': (H \delta_v, H\delta_{v'})=0$. Define
$$ \tilde\theta(H) = \left(\max_{u} \min_{v} {(H\delta_v, Hu)^2\over \|Hu\|^2}\right)^{-1}\,.$$
The claim is 
	$$ \min_H \tilde\theta(H) = \theta_L(G)\,.$$
	One direction is easy: Take $D$ achieving $\theta_L$ (WLOG, this $D$ has constant diagonal equal to $\theta_L$).
	Then let $D_1 = D/\theta_L$ and let $D_1 = H^* H$. To select $u$ introduce inner-product (here I make assumption
	that $D$ is invertible, need continuity argument for the general case):
		$$ (a,b)_D \eqdef (D_1 a, b) $$
	and notice that $D_1^{-1} J$ is self-adjoint in this metric, thus has $D_1-$orthonormal eigenbasis $v_1,\ldots,v_n$
	satisfying
		$$ Jv_i = \lambda_i D_1 v_i\,,$$
	which implies that only $\lambda_1$ is non-zero and thus
		$$ \lambda_1 = \lambda_{\max}(D_1^{-1} J) = \min \{\lambda: D_1 \succeq \lambda^{-1}J\} = \theta_L^{-1} $$
	and also (wlog we can flip $v_1$ if $(1,v_1)<0$):
		$$ (1,v_1) = \sqrt{\lambda_1}, D_1 v_1 =  1/\sqrt{\lambda_1} $$
	Thus $(\delta_v, v_1)_D = {1\over\sqrt{\lambda_1}}$ for all $v$ and hence taking $u=v_1$ proves
	$$ \tilde\theta(H) \le \theta_L(G) $$

	For the other direction (this is by A. Megretski) notice that having $\{H,u\}$ with WLOG $\|Hu\|=1$ implies that
	$I \succeq Hu u' H'$ and thus
		$$ H'H\succeq H'H u u' H' H \eqdef p'p\,,$$
		where $(p_i)^2 \ge {1\over \tilde\theta(H)}$. Thus defining
		$$ D=\Delta H'H \Delta\,,$$
		where $\Delta=\diag\{1/p_i\}$ we get $D\ge 11'$ and $D_{v,v} \le \tilde\theta(H)$.

\item Szegedy $\theta$-function:
	\begin{align} \theta_z(G) &\eqdef \max \{\tr JM: \tr M=1, M\succeq 0, M|_{E(G)} \le 0\}\\
		&= \min\{\max_v D_{v,v}: (D-J)\succeq 0, D|_{E(\bar G)}= 0, D_{v,v'} \ge 0\} \\
   		&= \min\{\lambda_{\max}(C): C=C^T, C|_{E(G)^c} = 1, C_{v,v'}\le 1\}
	\end{align}
\item Fractional chromatic number:
	$$ \chi^*(G) \eqdef \min\{{c\over a}: G\to K_{c,a}\}\,.$$
\item Relations 1:
	$$ \bbalpha(G) \le \theta_S(G) \le \theta_L(G) \le \theta_z(G) \le \chi^*(\bar G) \le \chi(\bar G)$$
	The last two is proved by applying~\eqref{eq:cdp2xx},~\eqref{eq:knes_dat} to graph-hom $\bar G \to K_m,
	m=\chi(\bar G)$ or $\bar G\to K_{c,a}$.
\item Relations 2:
	\begin{align}
	\mbox{$G$-v.trans.:} \quad \bbalpha(G) \bbalpha(\bar G) &\le |V(G)| \\
		\chi^*(G) &\ge \bbalpha(\bar G) \\
	\mbox{$G$-v.trans.:} \quad \hskip 24pt\chi^*(G) &= {|V(G)|\over \bbalpha(G)} \\
	\theta_L(G) \theta_L(\bar G) &\ge |V(G)|\\
	 \mbox{$G$-v.trans.:} \quad \theta_L(G)\theta_L(\bar G) &= |V(G)| \\
		\theta_S(G) \theta_z(\bar G) &\ge |V(G)| \label{eq:ts_lb1} \\
	 \mbox{$G$-v.trans.:} \quad \theta_S(G) \theta_z(\bar G) &= |V(G)|\\
	\end{align}	   
		Proof: The first is by averaging over automorphism $\gamma$ of $|\gamma(I)\cap K| \le 1$ for and indep.
		set $I$ and clique $K$. The following two pairs are proved similarly. Take $D$ achieving $\theta_z(\bar
		G)$ then $M=aD$ is a
		a feasible assignment in the (primal) for $\theta_S(G)$ with $a={1\over |G| (\max_v D_{v,v})}$. This
		proves inequality~\eqref{eq:ts_lb1}.
		Conversely, if $G$ is vertex transitive, then WLOG optimal $M$ achieving $\theta_S(G)$ has $1$ as an eigenvector and has a constant
		diagonal. Then $D={|G|^2 \over \tr JM} M$ is a feasible in the dual for $\theta_z(\bar G)$ and thus
		$\theta_S(G)\theta_z(\bar G) \le |G|$.
\item	A different kind:
	\begin{align} 
	 \mbox{$G$-v.trans.:}\quad  \theta_S(G) \theta_L(\bar G) &\le |V(G)| \label{eq:ts_ub2}
	\end{align} 
		Proof:\footnote{Inequality~\eqref{eq:ts_ub2} also holds more generally for Delsarte's assoc. schemes but
		different proof is needed.} 
			$|V(G)| = \theta_L(G) \theta_L(\bar G) \ge \theta_S(G) \theta_L(\bar G) $
\item Regular (constant degree) graphs:
	\begin{align} 
		\mbox{$G$-regular:}\quad \theta_L(G) \le \theta_z(G) &\le -|V(G)| {\lambda_n\over \lambda_1 - \lambda_n}\\
	   	\mbox{... and $G$-edge trans.:}\quad\theta_L(G) =\theta_z(G) &= -|V(G)| {\lambda_n\over \lambda_1 - \lambda_n}\,,
	\end{align}
	where $\lambda_1 \ge \cdots \ge \lambda_n$ are eigenvalues of $A_G$, $\lambda_1 = \mbox{degree~}G$.
		Proof: Just take $C=J-aA_G$ in~\eqref{eq:tl_min}; the best $a={|G|\over \lambda_1 - \lambda_n}$. By symmetry this is tight in the edge-transitive case.
\item Some info about the pentagon:
	\begin{align} \bbalpha(C_5)&=\bbalpha(\bar C_5) = 2\\
	   \theta_S(C_5) &= \theta_S(\bar C_5) = \theta_L(C_5) = \theta_L(\bar C_5) = \sqrt{5}\,.
	\end{align}
\item Relation between $\chi(G)$ and $\bbalpha(G)$: For vertex-transitive graph $G$\footnote{For general $G$ replace $|G|\over
\bbalpha(G)$ with $\sup{|G'|\over \bbalpha(G')}$ sup over all induced subgraphs.}
\begin{equation}\label{eq:alphachi}
	G\mbox{-v.t.:~~} {|G|\over \bbalpha(G)} \le \chi(G) \le {|G| \ln |G|\over \bbalpha(G)}\,.
\end{equation}
One direction is by no-homomorphism lemma applied to $G\to K_{\chi}$. The other direction is by covering $|G|$ with $m$
randomly chosen maximal independent sets (by transitivity each point is covered with the same probability, then taking
$m={|G|\ln |G|\over \bbalpha(G)}$ is sufficient by union bound).
\end{itemize}
}

The graph homomorphism $f: X\to Y$ is a map of vertices of $X$ to vertices of $Y$ such that endpoints of each edge of
$X$ map to the endpoints of some edge in $Y$. If there exists any graph homomorphism between $X$ and $Y$ we will write
$X\to Y$. The problem of finding $f:X\to Y$ is known as $Y$-coloring problem.

For establishing properties of graph homomorphisms it is convenient to introduce \textit{homomorphic
product}~\cite{BM95}\footnote{Note that~\cite{BM95} instead defines hom-product $X\circ Y$ which corresponds to $\overline{X\ltimes Y}$.}: graph
$X\ltimes Y$ is a simple graph with vertices $V(X)\times V(Y)$ and $(x_1,y_1)\sim (x_2,y_2)$ if $x_1=x_2$ or $x_1 \sim
x_2, y_1\not\sim y_2$.
From~\eqref{eq:lova_0} and definition of $X\ltimes Y$ we have:
$$ \bbalpha(X\ltimes Y) \le \theta_S(X\ltimes Y) \le \theta_L(X\ltimes Y) \le |V(X)| $$
and
$$ \bbalpha(X\ltimes Y)=|V(X)| \quad \iff \quad  X\to Y\,.$$

We overview some of the well-known tools for proving $X\not \to Y$:
\begin{itemize}
\item (No-Homomorphism Lemma~\cite{AC85}) If $X\to Y$ and $Y$ is vertex transitive then
	\begin{equation}\label{eq:nohom}
		{\bbalpha(X)\over |V(X)|} \ge {\bbalpha(Y)\over |V(Y)|}\,. 
	\end{equation}	
\item (Monotonicity of $\bar\bbalpha$) If $X\to Y$ then
	\begin{equation}\label{eq:cdp0}
		\bbalpha(\bar X) \le \bbalpha(\bar Y) 
	\end{equation}	
\item (Monotonicity of $\bar\theta$) If $X\to Y$ then\apxonly{\footnote{Easy proof: If $f:X\to Y$ and $D_{y,y'}$ is a
feasible dual for $\theta_S(\bar Y)$ in~\eqref{eq:thetas_alt} then $D'_{x,x'} = D_{f(x), f(x')}$ is a feasible dual for
$\theta_S(\bar X)$. Note that this mapping maps $J$ to $J$ so $D'\succeq J$ is obvious.}}
	\begin{align} \theta_L(\bar X) &\le \theta_L(\bar Y) \label{eq:cdp1}\\
	   \theta_S(\bar X) &\le \theta_S(\bar Y)\,.  \label{eq:cdp2}
	\end{align}	   
\apxonly{\begin{equation}
	   \theta_z(\bar X) \le \theta_z(\bar Y)\,.  \label{eq:cdp2xx}\end{equation}}
\item (Homomorphic product) If $X\to Y$ then
	\begin{align} 
		\theta_S(X \ltimes Y) &= |V(X)|  \label{eq:cdp3} \\
		   \theta_L(X \ltimes Y) &= |V(X)|\,. \label{eq:cdp4}
	\end{align}		   
\end{itemize}
Note that~\eqref{eq:cdp1}-\eqref{eq:cdp4} give necessary conditions for $X\to Y$. Although, generally not tight, these
conditions can be understood as elegant relaxations (semi-definite, fractional, quantum etc) of the graph homomorphism
problem, cf.~\cite{FL92,BM95,DR13,RM12,CMR14}.

\apxonly{\begin{remark} Here is a fuller story as per~\cite{FL92,CMR14}:
	$$ \begin{array}{rcrcrcr}
		X \to Y & \follows & X \stackrel{q}{\to} Y & \iff & \theta_S(X\ltimes Y) = |V(X)| 
				& \follows & \theta_S(\bar X) \le \theta_S(\bar Y)\\
			& & & & \Downarrow \hskip 40pt \\
			& & X \stackrel{B}{\to} Y & \iff &\theta_L(X\ltimes Y) = |V(X)| & \stackrel{\mbox{\cite{FL92}}}{\iff} & 
				\theta_L(\bar X) \le \theta_L(\bar Y) 
		\end{array} $$ 
\end{remark}}

Inequalities~\eqref{eq:nohom}-\eqref{eq:cdp4} are useful for showing $X\not \to Y$. If $X\not \to Y$ it is natural to
ask for a quantity measuring to what extent $X$ fails to homomorphically map into $Y$. One such quantity is
$\bbalpha(X\ltimes Y)$, since
\begin{equation}\label{eq:alpha_int}
	\bbalpha(X\ltimes Y) = \max\{|V(G)|: G\mbox{-- induced subgraph of $X$ s.t.~} G\to Y\}\,.
\end{equation}
Indeed, by construction any independent set $S$ in $X\ltimes Y$ has at most one point in each fiber $\{x_0\} \times Y$ and
thus projection $V(G)\eqdef\mproj_1(S)$ onto $X$ always yields an induced subgraph $G\subset X$ satisfying $G\to Y$.
With~\eqref{eq:alpha_int} in mind, the next set of results will allow us to assess the degree of failure of $X\not \to Y$.

\begin{theorem} 
If $X$ is vertex transitive, then
\begin{align}
	\bbalpha(X\ltimes Y) &\le |V(X)| {\bbalpha(\bar Y) \over \bbalpha(\bar X)} \label{eq:alpha_b1}\\
	\theta_S(X \ltimes Y) &\le |V(X)| {\theta_S(\bar Y)\over \theta_S(\bar X)} \label{eq:thetas_b1}\\
	\theta_L(X \ltimes Y) &\le |V(X)| {\theta_L(\bar Y)\over \theta_L(\bar X)} = \theta_L(X) \theta_L(\bar Y)\,. 
		\label{eq:thetal_b1}
\end{align}
	If $Y$ is vertex transitive, then
\begin{align} 
	\bbalpha(X\ltimes Y) &\le |V(Y)| {\bbalpha(X)\over \bbalpha(Y)} \label{eq:alpha_b2}\\
	\theta_S(X \ltimes Y) &\le |V(Y)| {\theta_S(X)\over \theta_S(Y)} \label{eq:thetas_b2}\\
	\theta_L(X \ltimes Y) &\le |V(Y)| {\theta_L(X)\over \theta_L(Y)} = \theta_L(X) \theta_L(\bar Y)\,.
		\label{eq:thetal_b2}
\end{align}	
\end{theorem}
\begin{remark} One may view~\eqref{eq:alpha_b1}-\eqref{eq:thetal_b1}
as a quantitative version of criteria~\eqref{eq:cdp0}-\eqref{eq:cdp2} and~\eqref{eq:alpha_b2}-\eqref{eq:thetal_b2} as
a quantitative version of no-homomorphism lemma~\eqref{eq:nohom}. 
Note also that the right-most versions of~\eqref{eq:thetal_b1} and~\eqref{eq:thetal_b2} hold without any transitivity
assumptions~\cite[Theorem 17]{BM95}: For any $X,Y$ 
	\begin{equation}\label{eq:thetal_b1a}
		\theta_L(X \ltimes Y) \le \theta_L(X) \theta_L(\bar Y)\,.
	\end{equation}
\end{remark}

\begin{proof} The proof relies on the following simple observation:
The strong product $X \boxtimes \bar Y$ 
	of $X$ and $\bar Y$ -- is a subgraph of $X \ltimes Y$. Thus by edge-monotonicity:
	$$ \bbalpha, \theta_S, \theta_L(X\ltimes Y) \le \bbalpha,\theta_S,\theta_L(X\boxtimes \bar Y)\,.$$
\apxonly{Note: for convenience here are edges of $X\ltimes Y$ and $X\boxtimes \bar Y$:
\begin{align} X\ltimes Y: & (x,y)\sim (x',y'): \,\, (x\sim x', y\not\sim y')\mbox{~or~}(x=x',y\neq y')\\
   X\boxtimes \bar Y: & (x,y)\sim (x',y'): \,\,
   			(x\sim x', y\not\sim y')\mbox{~or~}(x=x',y\neq y',\alert{y\not\sim y'}) 
\end{align}			
}%

From here the results on $\bbalpha$ and $\theta_S$ follow from Lemma~\ref{th:sprod} with $G=X$, $H=\bar Y$
(for~\eqref{eq:alpha_b1} and \eqref{eq:thetas_b1}) or $G=\bar Y$ and $H=X$ (for~\eqref{eq:alpha_b2}
and~\eqref{eq:thetas_b2}).
For $\theta_L$ the equality parts of~\eqref{eq:thetal_b1} and~\eqref{eq:thetal_b2} follow from 
	the results of Lov\'asz~\eqref{eq:lova_1} and~\eqref{eq:lova_2}.
\qed
\end{proof}

One of the classically useful methods in coding theory is the Elias-Bassalygo reduction: From a given code in $\FF_2^n$
one selects a large subcode sitting on a Hamming sphere of a given radius. One then bounds minimum distance (or other) 
parameters for the packing problem in the Johnson graph $J(n,d,w)$. It so happens that taking a simple dual certificate for 
$\theta_S(J(n,d,w))$ and transporting the bound back to the full space results in excellent bounds, which are hard (but
possible -- see Rodemich theorem in~\cite[p. 27]{delsarte1994application}) to obtain by direct SDP methods in the full space.  Succinctly, we may summarize this as follows: If $G'$ is
an induced subgraph of a vertex transitive $G$ then
$$ \bbalpha(G), \theta_S(G), \theta_L (G) \le {|V(G)|\over |V(G')|} \bbalpha(G'), \theta_S(G'),\theta_L(G')
\quad\mbox{resp.} $$

Here is a version of the similar method for the graph-homomorphism problem and for the problem of finding independent
sets in $G\boxtimes H$:
\begin{proposition}\label{th:elbas_prod} Let $G$ be a vertex transitive graph and $G'$ its induced subgraph. Then
\begin{align} 
	\bbalpha(G\boxtimes H) &\le {|V(G)|\over |V(G')|} \bbalpha(G'\boxtimes H) \label{eq:elbas_alpha} \\
	\theta_S(G\boxtimes H) &\le {|V(G)|\over |V(G')|} \theta_S(G'\boxtimes H) \label{eq:elbas_theta} 
\end{align}
and same for $\theta_L$.
\end{proposition}

\begin{proof} Let $\Gamma$ be the group of automorphisms of $G$. The action of $\Gamma$ naturally extends to the action
on $G\boxtimes H$ via:
$$ \gamma(g,h) \eqdef (\gamma(g), h)\,.$$
Let $S$ be the maximal independent set of $G\boxtimes H$. Consider the chain:
\begin{align} \bbalpha(G' \boxtimes H) &\ge {1\over |\Gamma|} \sum_{\gamma \in\Gamma} |\gamma(S) \cap G' \boxtimes H| 
		\label{eq:eb1} \\
	&= {1\over |\Gamma|} \sum_{\gamma \in\Gamma, g,g',h} 1\{\gamma(g)=g'\} 1\{(g',h)\in S\} 1\{g\in
	G'\}\label{eq:eb2}\\
	&= {|S| \, |V(G')| \over |V(G)|} \label{eq:eb3}\,,
\end{align}
where~\eqref{eq:eb1} follows since each $\gamma(S) \cap G' \boxtimes H$ is an independent set of $G'\boxtimes
H$,~\eqref{eq:eb2} is obvious,
and~\eqref{eq:eb3} is because by the transitivity of the action of $\Gamma$: $\sum_\gamma 1\{\gamma(g)=g'\} =
{|\Gamma|\over V(G)}$. Clearly, \eqref{eq:eb3} is equivalent to~\eqref{eq:elbas_alpha}.

For~\eqref{eq:elbas_theta} let $M=(M_{g_1h_1,g_2h_2}, g_1,g_2\in G, h_1,h_2 \in H)$ be the maximizer
in~\eqref{eq:ts_max}. Symmetrizing over $\Gamma$ if necessary we may assume that
\begin{align}\label{eq:ebt_1}
	M_{g h_1, g h_2} &= M_{g'h_1, g' h_2}  & \forall g,g'\in G, h_1,h_2 \in H\\
	M_{g_1 h_1, g_2 h_2} &= M_{\gamma(g_1)h_1, \gamma(g_2) h_2} & \forall g_1,g_2\in G, h_1,h_2 \in H, \forall
	\gamma \in \Gamma
\end{align}
Last equation also implies that the subspace spanned by vectors $1_G\otimes (\cdot)$ is an eigenspace of $M$. 
Here and below $1_G, 1_H$
are all-one vectors of dimensions $|V(G)|$ and $|V(H)|$ respectively. And $1_{G'}$ is a zero/one vector of dimension
$|V(G)|$ having ones in coordinates corresponding to vertices in $G'$. 

Set
$$ \tilde M_{g_1 h_1, g_2 h_2} = {|V(G)|\over |V(G')|} M_{g_1 h_1, g_2 h_2} \quad \forall g_1,g_2\in G', h_1,h_2\in
H\,.$$
One easily verifies that $\tilde M$ is a feasible choice for the primal program~\eqref{eq:ts_max} for
$\theta_S(G'\boxtimes H)$. To compute $\tr J\tilde M$ we notice that
\begin{equation}\label{eq:ebt_2}
	\tr J\tilde M = {|V(G)|\over |V(G')|} \left( M 1_{G'} \otimes 1_H, 1_{G'} \otimes 1_H\right)\,,
\end{equation}
where $(\cdot,\cdot)$ is a standard inner product on $\mreals^{|V(G)|} \otimes \mreals^{|V(H)|}$. Finally, observe that
orthogonal decomposition
$$ 1_{G'} \otimes 1_H = c 1_G \otimes 1_H + (1_{G'} - c1_G)\otimes 1_H\,, \qquad c= {|V(G')|\over |V(G)|} $$
remains orthogonal after application of $M$, cf.~\eqref{eq:ebt_1}. Therefore, we get by positivity $M\succeq 0$ that
$$ \left( M 1_{G'} \otimes 1_H, 1_{G'} \otimes 1_H\right) \ge c^2 (M 1_G \otimes 1_H, 1_G \otimes 1_H) = c^2 \tr JM\,,$$
which together with~\eqref{eq:ebt_2} completes the proof of~\eqref{eq:elbas_theta}.
\qed
\end{proof}

\begin{corollary} \label{th:elbas}
Let $X'$ and $Y'$ be induced subgraphs of $X$ and $Y$, respectively. If $X$ is vertex transitive then
$$ \bbalpha(X \ltimes Y) \le {|V(X)|\over |V(X')|} \bbalpha(X'\ltimes Y)\,.$$
If $Y$ is vertex transitive then
$$ \bbalpha(X\ltimes Y) \le {|V(Y)| \over |V(Y')|} \bbalpha(X\ltimes Y')\,.$$
\end{corollary}

\apxonly{\textbf{Remark:} For the case when $f:X\to Y$ exists, the proof says that (obviously) there is a restriction
$f':X'\to Y$ for any $X' \subset X$. The second statement is more interesting. It says that for any $Y'$ there is an
$X''$ of size $|X''| \ge {|Y'|\over |Y|}$ s.t. there is an $f'': X''\to Y'$, and the $X''$ can be taken to be $(\gamma
\circ f)^{-1} Y'$}

\begin{proof}[Lemma~\ref{th:sprod}]
\apxonly{Inequality~\eqref{eq:sproda} follows from Proposition~\ref{th:elbas_prod} by taking $G'$ to be the
maximal clique in $G$. Then $\bbalpha(G' \boxtimes H) = \bbalpha(H)$ and~\eqref{eq:sproda} follows.}
 We will give an explicit proof by exhibiting a choice of $\tilde C$ in~\eqref{eq:ts_min} for computing
$\theta_S(G\boxtimes H)$. 

Let $M$ be the optimal (primal) solution of~\eqref{eq:ts_max} for $\theta_S(\bar G)$ and let $C$ be the
optimal (dual) solution of~\eqref{eq:ts_min} for $\theta_S(H)$. We know:
$$ \tr JM = \theta_S(\bar G), \lambda_{\max}(C) =\theta_S(H)\,.$$
   and also from the vertex-transitivity of $G$ without loss of generality we may assume that 
   $$ M_{g,g} = {1\over |V(G)|}, M1 = {\tr JM\over |V(G)|} 1 \,,$$
   where $1$ is an all-one vector. We now define\footnote{This choice may appear mysterious, but notice that if we
   define $D=C-\lambda_{\max}(C)I-J$ and assuming $\lambda_{\max}(\hat C) = c_1$ we could write~\eqref{eq:ts_q3} as 
   $\hat D = c_2 M \otimes D$, which is more natural.}
   \begin{equation}\label{eq:ts_q3}
   	\hat C \eqdef c_1 I + c_2 M \otimes(C-\lambda_{\max}(C)I - J) + J\,,
\end{equation}   
   where as before $J$ denotes the square matrix of all ones (of different dimension depending on context) and
   \begin{equation}\label{eq:ts_q2a}
   	c_1 \eqdef {\lambda_{\max}(C) |V(G)|\over \tr JM}, \quad c_2 = {|V(G)|^2\over \tr JM}\,. 
\end{equation}   
   We will prove that $\hat C$ is a feasible choice in the (dual) problem~\eqref{eq:ts_min} for $\theta_S(G\boxtimes H)$.
   Then we can conclude that since $M\succeq 0$ and $C-\lambda_{\max}(C) I \preceq 0$ that
   \begin{align}\label{eq:ts_q1}
   	\hat C \preceq c_1 I - c_2 M \otimes J + J = c_1 I - (c_2 M - J) \otimes J \preceq c_1 I 
\end{align}   
   since by construction $c_2 M - J\succeq 0$ (recall that $M$ and $J$ commute). Thus,
   $$ \lambda_{\max}(C) \le c_1 = {|V(G)| \theta_S(H) \over \theta_S(\bar G)}$$
   proving~\eqref{eq:thetas_b1}. To verify that $\hat C$ is feasible dual assignment, we need to show
   \begin{equation}\label{eq:ts_q2}
   	\hat C_{gh, g'h'} \ge 1 \quad \forall g,h,g',h': 
   		\begin{cases} g=g',h=h', &\mbox{or~}\\
		g\neq g', g\not\sim g', &\mbox{or~}\\
		h\neq h', h\not\sim h'\,, \end{cases}
\end{equation}		
   which follows since $E(G\boxtimes H)^c$ consists of all self-loops and edges connecting pairs that are non-adjacent
   (and non-identical) in either $G$ or $H$-coordinate.

   To verify~\eqref{eq:ts_q2} we recall that $M$ and $C$ satisfy
\begin{align} 
   M_{g,g'}&\ge 0 \quad \forall g,g' \label{eq:ts_q0a}\\
	M_{g,g'} &= 0 \quad \forall g\not\sim g'\mbox{~and~} g\neq g' \label{eq:ts_q0aa}\\
   C_{h,h'} &\ge 1 \quad \forall h\not\sim h'  \label{eq:ts_q0b}
\end{align}
   Then verification proceeds in a straightforward manner. For example, in the first case in~\eqref{eq:ts_q2} we have
   \begin{align*} 
   	\hat C_{gh,gh} &= c_1 + {c_2\over |V(G)|}(C_{h,h} - \lambda_{\max}(C) - 1) + 1 \\
		&\ge c_1 - {c_2 \lambda_{\max}(C)\over |V(G)|} + 1 = 1 
\end{align*}   
   because $C_{h,h}\ge 1$ and by~\eqref{eq:ts_q2a}. The two remaining cases are checked similarly. 

\apxonly{Alternatively, for~\eqref{eq:thetas_b1} we could take any
$\Gamma$ from the feasible set of the primal~\eqref{eq:ts_max} for $\theta_S(X\ltimes Y)$, take $\bar M$ to be the optimal in
the primal for $\theta_S(\bar X)$ and show that 
$$ N_{y,y'} = c \sum_{x,x'} \Gamma_{xy, x'y'} \bar M_{x,x'} $$
is a feasible for $\theta_S(\bar Y)$ with $\tr JN \ge \cdots$. }

	\qed
\end{proof}

\apxonly{
\subsection{Random comments about graph homomorphisms}
\begin{itemize}
\item Graph hom. $X\to Y$ does not mean $X$ is a subgraph of $Y$. For example, $\square \to K_2$ but $\square \not
\hookrightarrow K_2$. However, $\nabla
\to Y \iff \nabla \hookrightarrow Y$ (and also for any clique).
\end{itemize}
}

\section{Proofs of main results}\label{sec:hamming}

Before going into details of the proofs, we make a clarifying remark. Our main goal is to improve the simple
bound~\eqref{eq:ccb}, which (we remind) was obtained by noticing that independent sets of $H(k,\alpha k)$ transform
under $(\alpha,\beta)$-maps into independent sets of $H(n,\beta n)$. The improvement comes by noticing that
$(\alpha,\beta)$-maps also transform any matrix $M$ in~\eqref{eq:ts_max} feasible for $H(k,\alpha k)$ into a matrix $M'$
feasible for $H(n,\beta n)$. Good feasible matrices for $H(k,\alpha k)$ were found previously
in~\cite{AS01,navon2005delsarte}. 
We proceed to formal details.

\begin{proof}[Theorems~\ref{th:main1} and~\ref{th:main2}]
Clearly, it is sufficient to prove Theorem~\ref{th:main2}. We quote the following results of Kleitman~\cite{DK66},
McEliece et al~\cite{MRRW77} and the joint lower bound of Samorodnitsky~\cite{AS01} and 
Navon-Samorodnitsky~\cite{navon2005delsarte}:
\begin{align} {1\over m} \log \bbalpha(\bar H(n, \lambda m)) &= h(\lambda/2) + \delta_m(\lambda)\label{eq:tkl}\\
	{1\over m} \log \theta_S(H(m, \lambda m)) &\le R_{LP2}(\lambda) + \delta_m(\lambda)\label{eq:tmrrw}\\
	{1\over m} \log \theta_S(H(m, \lambda m)) &\ge R_{Sam}(\lambda) - \delta_m(\lambda)\label{eq:tsam}\,,
\end{align}
where the remainder term $\delta_m(\lambda)\to0$ uniformly on compacts in $\lambda \in (0,1]$. 

Take $\alpha,\beta \ge
\epsilon$ and $k,n \in \mathbb{Z}_+$. Define $\delta_m = \sup_{\lambda \in[\epsilon,1]}\delta_m(\lambda)$. Assume
that~\eqref{eq:tm2_1} holds. Consider arbitrary $f:\FF_2^k\to\FF_2^n$. Notice that if $S$ is a set which does not
contain any pair satisfying~\eqref{eq:tm2_2a}, then the set
$$ \{(x,y): x\in S, y=f(x)\} $$
is an independent set of $\bar H(k, \alpha k) \ltimes \bar H(n, \beta n)$. (This is also clear
from~\eqref{eq:alpha_int} as $f$ defines a homomorphism $S\to \bar H(n,\beta n)$ if $S$ is viewed as induced subgraph of
$\bar H(k, \alpha k)$.) Thus it is sufficient to show
$$ \bbalpha(\bar H(k, \alpha k) \ltimes \bar H(n, \beta n)) \le 2^{k(1-\epsilon) + n\delta_n + k\delta_k} $$
This follows from the following chain:
\begin{align} \bbalpha(\bar H(k, \alpha k) \ltimes \bar H(n,\beta n)) 
	&\le 2^k {\theta_S(H(n,\beta n))\over \theta_S(H(k, \alpha k))}\label{eq:tm22}\\
	&\le 2^{k + nR_{LP2}(\beta) -k R_{Sam}(\alpha) + n\delta_n + k\delta_k}\label{eq:tm23}\\
	&\le 2^{k(1-\epsilon) + n\delta_n + k\delta_k}\label{eq:tm24}\,,
\end{align}
where~\eqref{eq:tm22} is from~\eqref{eq:thetas_b1},~\eqref{eq:tm23} is from~\eqref{eq:tmrrw} and~\eqref{eq:tsam}
and~\eqref{eq:tm24} is from~\eqref{eq:tm2_1}. 

If instead of~\eqref{eq:tm2_1} the pair $(\alpha,\beta)$ satisfies~\eqref{eq:tm2_2} then the argument is the same except
in~\eqref{eq:tm22} we should apply~\eqref{eq:alpha_b2} and~\eqref{eq:tkl} to get:%
\footnote{%
Note that another result of Samorodnitsky~\cite[Proposition 1.2]{AS_Delsarte} shows that up to factors $2^{o(m)}$ we
have%
$$\theta_L(\bar H(m, \lambda m))
\approx \theta_S(\bar H(m, \lambda m)) \approx \bbalpha(\bar H(m,\lambda m)) = 2^{m h(\lambda/2) + o(m)}\,. $$ 
Therefore here
we could still operate with $\theta_S$ only and apply~\eqref{eq:thetas_b2}. We chose to use $\bbalpha$'s because
Kleitman's theorem~\eqref{eq:tkl} has explicit non-asymptotic form and thus for finite $k,n$ results in a better bound.
}
$$ \bbalpha(\bar H(k, \alpha k) \ltimes \bar H(n,\beta n)) \le 2^{k + n(1- n h(\beta/2) +\delta_n) - k (1-h(\alpha/2)
-\delta_k)} $$
and the rest of the proof is the same.

Finally, to show the statement about the number of pairs satisfying~\eqref{eq:tm2_2a} define a graph $G$ with vertices $\FF_2^k$ and $x\sim x'$ 
if~\eqref{eq:tm2_2a} holds. We have already shown
$$ \bbalpha(G) \le 2^{k(1-\epsilon) + n\delta_n + k\delta_k}\,.$$
Then from Turan's theorem we have
$$ |E(G)| \ge {|V(G)|\over 2} \left({|V(G)|\over \bbalpha(G)} - 1\right) \ge  2^{k(1+\epsilon-\delta_k)-n\delta_n} $$
(after enlarging $\delta_k$ slightly).
\qed
\end{proof}

\medskip

\begin{proof}[Theorem~\ref{th:main3}] The argument follows step by step the proof of Theorem~\ref{th:main3} except that
at~\eqref{eq:tm22} we use the (almost) exact value of $\theta_S(H(n,d))$ for $d>n/2$ found in the Lemma below.
\end{proof}

\begin{lemma}\label{th:plot} For any $\lambda\in(1/2,1)$ there exists $\delta_n(\lambda)\ge0$ s.t.
\begin{align} {2\lambda\over 2\lambda - 1} -\delta_n(\lambda) &\le \theta_S(H(n, \lfloor \lambda n\rfloor)) 
			\label{eq:p1}\\
				&\le {2\lambda \over 2\lambda -1}\,. \label{eq:p2}
\end{align}
Furthermore, $\delta_n(\lambda)\to 0$ uniformly on compacts of $(1/2, 1)$.
\end{lemma}

\begin{proof} We need to introduce the standard definitions from linear programming bounds in coding theory,
cf.~\cite[Ch. 17]{MS1977}. Any polynomial $f(x) \in \mreals[x]$ of degree $\le n$ can be represented as
$$ f(x) = 2^{-n} \sum_{j=0}^n \hat f(j) \krawt_j(x)\,,$$
where Krawtchouk polynomials are defined as
\begin{equation}\label{eq:krawt}
	\krawt_j(x) \eqdef \sum_{k=0}^n (-1)^j {x \choose k} {n-x \choose
j-k} 
\end{equation}
and, for example, $\krawt_0(x)=1$, $\krawt_1(x)=n-2x$.

It is a standard result~\cite[Theorem 3]{AS79} that for the Hamming graphs the semidefinite program~\eqref{eq:ts_max}
becomes a linear program. We put it here in the following form:
\begin{align} \lefteqn{\theta_S(H(n, d))}&\nonumber\\
&= \max\left\{ {\hat f(0)\over f(0)}: \hat f \ge 0, f(x)=0, x\in[1,d]\cap\mathbb{Z}, f(x) \ge 0,
x\in[0,n]\cap \mathbb{Z}\right\}
			\label{eq:ts_ham1}\\
		     &= \min\left\{ 2^n {g(0)\over \hat g(0)}: \hat g \ge 0, g(x)\le 0, x \in[d+1, n]\cap \mathbb{Z},
		     \hat g(0)>0 \right\} 				\label{eq:ts_ham2}
\end{align}
\apxonly{$$ =\min\left\{ \max_{\omega} \hat h(\omega): h(0) \ge 1, h(x) \ge 1, x\in[d+1,n]\cap \mathbb{Z}
		     \right\} $$}
where $f$ and $g$ are polynomials of degree at most $n$.

The upper bound~\eqref{eq:p2} is a standard Plotkin bound (see~\cite[Ch. 17, \S4]{MS1977}): taking $g(x) =2(d+1-x)$ we notice that 
$$ g(x) = \krawt_1(x) + (2(d+1)-n)\krawt_0(x)\,. $$
Thus $\hat g(0) = 2^n(2d+2-n)$ and we get for $d=\lfloor \lambda n\rfloor$
$$ \theta_S(H(n,d)) \le {2(d+1)\over 2d+2-n} \le {2\lambda \over 2\lambda -1}\,.$$

For the lower bound~\eqref{eq:p2} we assume that $d=\lfloor \lambda n\rfloor$,
\begin{equation}\label{eq:p6}
	\lambda \in[1/2-\epsilon, 1-\epsilon]  
\end{equation}
and $n$ is sufficiently large (for all small $n$ we may take $\delta_n(\lambda)={2\lambda\over 2\lambda-1}$). We first
consider the case of $d$ -- odd.

Consider the polynomial $f(x)=2^{-n} \sum_{\omega=0}^n \krawt_\omega(x)$ with coefficients given by
\begin{equation}\label{eq:p8}
	\hat f(\omega) = \krawt_0(\omega) + r {n \choose d+1}^{-1} \krawt_{d+1}(\omega)\,, \omega = 0, 1, \ldots, n 
\end{equation}
where $r\in(0,1)$ is to be determined. To compute values of this polynomial, we employ the orthogonality relation for
Krawtchouk polynomials, cf.~\cite[(34)]{KL01}:
$$ \sum_{\omega=0}^n \krawt_\ell(\omega) \krawt_\omega(x) = 2^n 1\{x=\ell\}, \qquad \forall x,\ell \in[0,n]\cap
\mathbb{Z} $$
From here we get
$$ f(x) = \begin{cases} 1, & x=0\\
			r {n \choose d+1}^{-1}, & x = d+1\\
			0, &\mbox{all other $x\in [0,n]\cap \mathbb{Z}$}\,.
	\end{cases} $$
We note that this $f(x)$ was guessed by studying Levenshtein's codes that attain Plotkin bound~\cite[Chapter 2, Theorem
8]{MS1977}. 

To verify that $f(x)$ is a (asymptotically!) feasible solution of~\eqref{eq:ts_ham1} we need to check $\hat f(\omega)\ge 0$. First, 
let $m = n-d-1 \le {n/2}$ and notice that~\cite[(31)-(32)]{KL01}
$$ \krawt_{d+1}(\omega) = (-1)^{\omega} \krawt_{m}(\omega) = (-1)^{(n-\omega)+(m-n)} \krawt_{m}(n-\omega)\,.$$
Therefore, since $n-m$ is even it is sufficient to verify
\begin{equation}\label{eq:p3}
	(-1)^\omega \krawt_m(\omega) \ge -{1\over r} {n\choose m} 
\end{equation}
for all $\omega \in [0,n/2]\cap\mathbb{Z}$. For $\omega = 0$ this is obvious, for 
$\omega = 1$ we have~\cite[(13)]{KL01}
$$ \krawt_m(1) = {n-2m\over n} {n\choose m} $$
and thus taking 
$$ r={n\over n-2m} $$
makes~\eqref{eq:p3} hold at $\omega=1$. 

It is known that $\krawt_m(x)$ has $m$ real roots with the smallest root
$x_1$ satisfying~\cite[(71)]{KL01}
\begin{equation}\label{eq:p4}
	x_1 \ge {n\over 2} - \sqrt{m(n-m)}\,.
\end{equation}
Therefore, polynomial $\krawt_m(x)$ is decreasing on $(-\infty, x_1]$ and
hence~\eqref{eq:p3} must also hold for all odd $\omega \in[1,x_1]$ (for even $\omega \le x_1$, inequality~\eqref{eq:p3} holds
just by considering the signs). In view of~\eqref{eq:p4} we only need to show~\eqref{eq:p3} for 
$ \xi n \le \omega \le n/2$, where
\begin{equation}\label{eq:p5}
	\xi = {1\over 2} - \sqrt{\lambda(1-\lambda)}\,. 
\end{equation}
In this range, we will show a stronger bound
\begin{equation}\label{eq:p3a}
	|\krawt_m(x)| \le {1\over r} {n\choose m}\,.
\end{equation}

The following bound is well known~\cite[(87)]{KL01}\footnote{To get an explicit estimate on $\delta_n(\lambda)$
in~\eqref{eq:p1}, we could use the non-asymptotic bound in~\cite[Lemma
4]{YP13}, which also holds for $\omega < \xi n$.}:
$$ |\krawt_m(\omega)| \le 2^{n\over 2} {n\choose m}^{1\over 2} {n \choose \omega}^{-{1\over2}}\,.$$
Note that by the constraint~\eqref{eq:p6}  
$\xi$ is bounded away from $0$ and thus we can estimate
$$ {n \choose m} {n \choose \omega}^{-1} \le 2^{n (h(\lambda) - h(\xi) + \delta'_n)} $$
for some sequence $\delta'_n$ that only depends on $\epsilon$. Thus comparing the exponents on both sides
of~\eqref{eq:p3} we see that it will hold provided that 
\begin{equation}\label{eq:p7}
	r^{-2} \ge 2^{n (1-h(\lambda) + h(\xi)) + n\delta'_n}\,.
\end{equation}
But notice that by~\eqref{eq:p5} the exponent in parenthesis is exactly the gap between the Gilbert-Varshamov bound
$1-h(1-\lambda))$ and the first linear programming bound $R_{LP1}(1-\lambda)$, cf.~\eqref{eq:rlp1}. 
There exists $\epsilon'>0$ separating these two bounds for all $\lambda$'s in~\eqref{eq:p6}. Thus, the right-hand side
of~\eqref{eq:p7} is exponentially decreasing $2^{-\epsilon' n + n\delta'_n}$ and hence for sufficiently large $n$ it
must hold. This completes the proof that $f(x)$ in~\eqref{eq:p8} is a feasible choice in~\eqref{eq:ts_ham1}. Therefore,
we have shown that for all $n$ sufficiently large 
$$ \theta_S(H(n,d)) \ge 1+r = {2d+2\over 2d+2-n} $$
if $d$ is odd and 
$$ \theta_S(H(n,d)) \ge \theta_S(H(n,d+1)) \ge {2d+4\over 2d+4-n} $$
if $d$ is even.
\qed
\end{proof}

\smallskip
\begin{proof}[Corollaries~\ref{th:cor2} and~\ref{th:cor3}] 
Assume to the contrary that one found $\alpha,\beta$ and $u_1,\ldots,u_k,v_1,\ldots,v_n\in\PP^{m-1}$ s.t. there is no
hyperplane satisfying~\eqref{eq:ab_hlin}. Then as explained in Section~\ref{sec:lincode} below (see~\eqref{eq:ab_linx}),
there is an independent set of size $2^m$ in $\bar H(k,\alpha k) \boxtimes H(n, \beta n)$. By inspecting the proofs
of Theorem~\ref{th:main2} and~\ref{th:main3} we notice that they prove three different upper bounds on 
$\bbalpha(\bar H(k,\alpha k) \boxtimes H(n, \beta n))$ that are equal to exponentiation of the left-hand sides
of~\eqref{eq:corA},~\eqref{eq:corB} and~\eqref{eq:corC} respectively. Therefore, $m$ cannot satisfy any
of~\eqref{eq:corA},\eqref{eq:corB} or~\eqref{eq:corC} -- a contradiction. \qed
\end{proof}

\apxonly{
\subsection{Knowledgebase for $H(n,d)$ and $J(n,d,w)$}

Some classical results:
\begin{itemize}
\item Kneser graph is the complement of the max-distance Johnson:
$$ K(n,w) = \bar J(n,2w-1,w)\,. $$
\item Erdos-Ko-Rado and its $\theta$-extensions:
	\begin{align} 
		\bbalpha=\theta_S&=\theta_L=\theta_z(K(n,w)) = {n-1\choose w-1}\\
		\lfloor {n\over w}\rfloor = \bbalpha(\bar K(n,w)) &\le \theta_S \le \theta_L=\theta_z(\bar
		K(n,w)) = {n\over w}\label{eq:knes_dat}\\
	   	 \theta_S&\le \theta_L(J(n,2w,w)) = {n\over w}\,. 
	\end{align}		 
\item Note: EKR also proves stuff for $\bar J(n, 2w-t,w)$ when $n\gg w$. And later Schrijver extended it to $\theta_L$
too. TODO!
\item Kneser conjecture:
	$$ \chi(\bar J(n, 2w-1, w)) = n-2w+2 $$
\end{itemize}

We review some of the well known asymptotic results on Hamming and Johnson graphs. First, Kleitman~\cite{DK66}
and Ahslwede-Katona show:
\begin{align} {1\over n} \log \bbalpha(\bar H(n, \delta n)) &= h(\delta/2) + o(1)\\
   {1\over n} \log \bbalpha(\bar J(n, \delta n, \xi n)) &= 
   	a h\left( b \over a\right) + 
			(1-a) 
	h\left( {\xi - b\over 1-a}  \right)
			\\
			& a={\xi - \delta/2\over 1-\delta}, b={(\xi - \delta/2)(2-\beta)\over2-2\beta}\\
\end{align}
In the first case the optimal set is $B(n, \delta n/2)$, in the second case it is $S(n, \xi n) \cap \{u+B(n, \theta
n)\}$ and $|u| = n(\xi - {\delta/2}+2\theta)$ (eccentric ball) and $\theta = {\delta\over2}{\xi - \delta/2 \over
1-\delta}$. 

Samorodnitsky showed~\cite{AS_Delsarte}:
\begin{align} {1\over n} \log \theta_L(\bar H(n, \delta n)) &= h(\delta/2) + o(1)\,,\\
   {1\over n} \log \theta_S(\bar H(n, \delta n)) &= h(\delta/2) + o(1)\,,\\
   {1\over n} \log \theta_L(H(n,\delta n)) &= 1-h(\delta/2) + o(1)\qquad\mbox{By Lovasz identity}
\end{align}

Next, the McEliece et al bounds~\cite{MRRW77} give\footnote{The bound for the Johnson graph can be improved for small
$\beta$ by appealing to Levenshtein's monotonicity, cf~\cite[Lemma 1.4]{AS01}, but we will not need this bound here.}
\begin{align} {1\over n} \log \theta_S(H(n, \delta n)) &\le R_{LP2}(\delta) + o(1)\\
{1\over n} \log \theta_S(J(n, \delta n, \xi n)) &\le R_{LP}(\delta,\xi) + o(1)\,,
\end{align}
where $R_{LP2}$ was defined in~\eqref{eq:rlp2} and 
$$ R_{LP}(\delta, \xi) = \begin{cases} 0, &\delta > 2\xi(1-\xi)\,,\\
				h\left({1-\sqrt{1-u^2}\over2}\right), &u=\sqrt{\delta^2-2\delta+4\xi(1-\xi)}-\delta
			\end{cases}
$$
Regarding the lower bounds on $\bbalpha$, the best ones (for binary case) are the 
Gilbert-Varshamov bounds (Turan's theorem, equivalently):
\begin{align} {1\over n} \log \bbalpha(H(n, \delta n)) &\ge 1-h(\delta) + o(1)\,,\\
{1\over n} \log \bbalpha(J(n, \delta n, \xi n)) &\ge R_{GV}(\delta,\xi) + o(1)\,,
\end{align}
where
$$ R_{GV}(\delta,\xi) = \begin{cases}
			0, &\delta > 2\xi(1-\xi)\,,\\
			h(\xi) -\xi h\left(\delta\over 2\xi\right) - (1-\xi) h\left(\delta\over 2(1-\xi)\right)\,,
			&\mbox{otherwise}
			\end{cases}\,.$$
Finally, for $\theta_S$ better bounds were found by Samorodnitsky~\cite{AS01}:
\begin{align} {1\over n} \log \theta_S(H(n, \delta n)) &\ge {1\over 2}(1-h(\delta) + R_{LP1}(\delta)) + o(1)\\
	{1\over n} \log \theta_S(J(n, \delta n, \xi n)) &\ge {1\over 2}(R_{GV}(\delta,\xi) + R_{LP}(\delta,\xi)) + o(1)\,,
\end{align}
where $R_{LP1}$ was defined in~\eqref{eq:rlp1}.}

\section{Discussion and open problems}\label{sec:disc}

\begin{figure}[t]
\centering
\subfigure[$\rho=3$]{\label{fig:eval1}\includegraphics[width=0.4\textwidth]{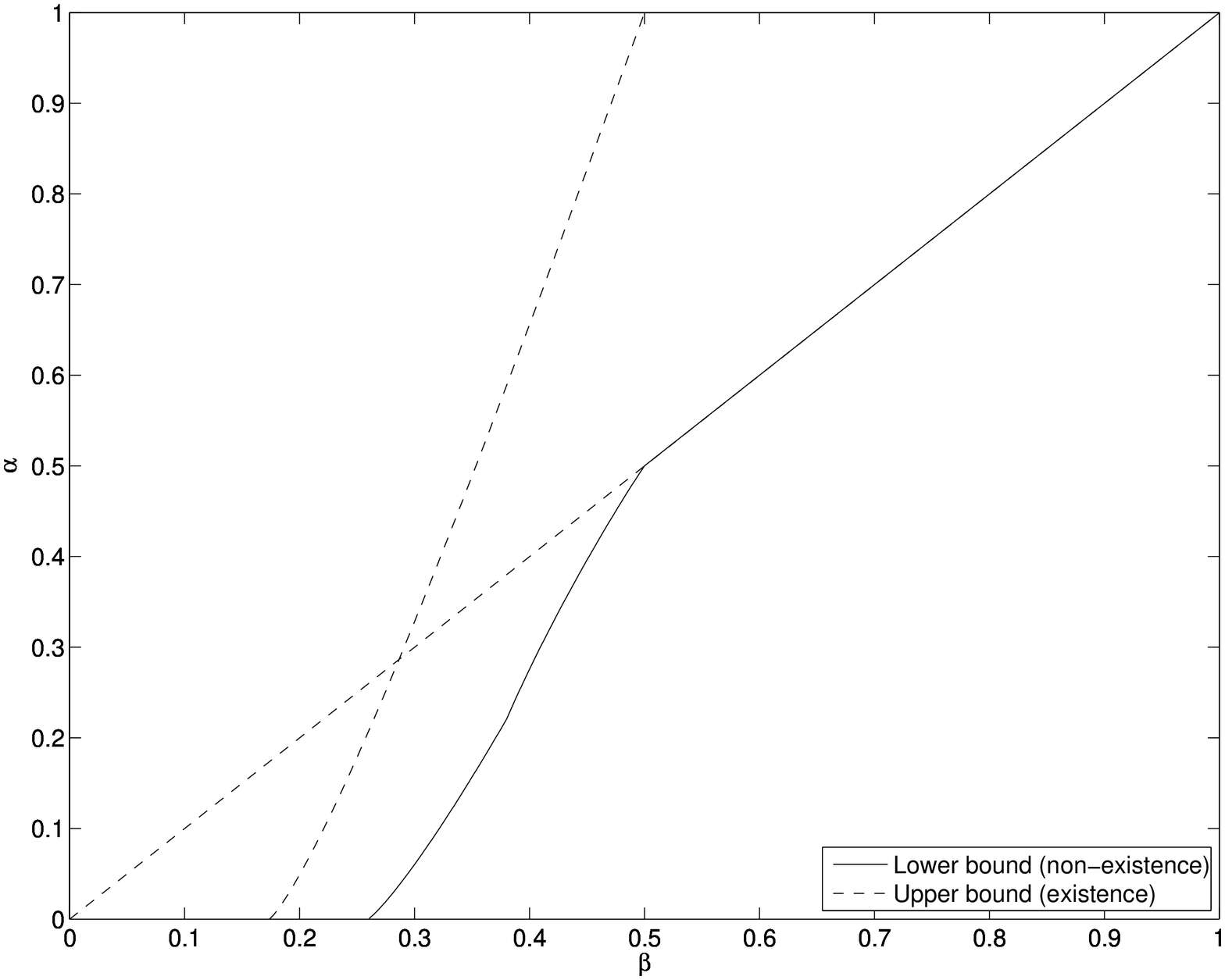}}
\subfigure[$\rho={1\over3}$]{\label{fig:eval2}\includegraphics[width=0.4\textwidth]{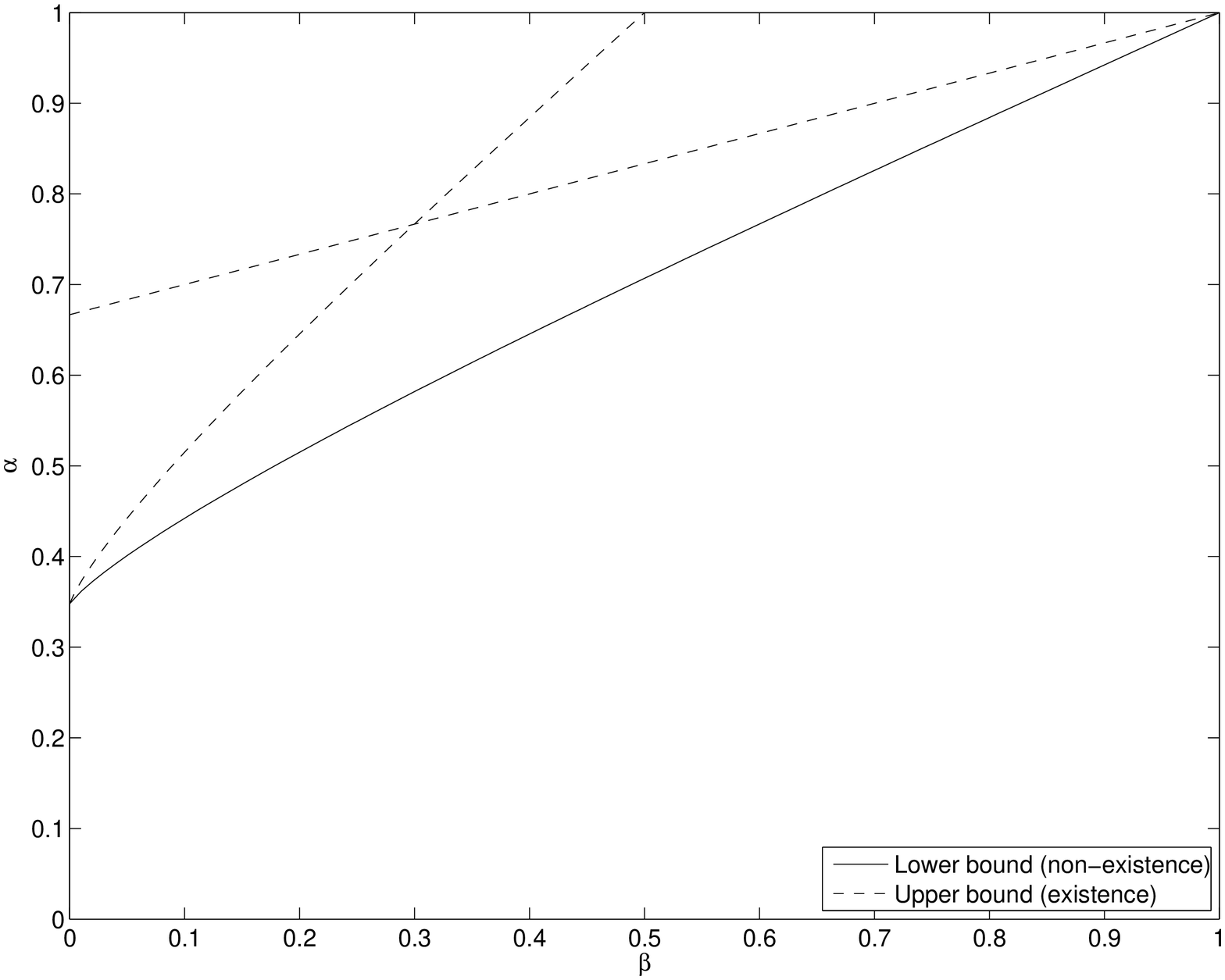}}
\caption{Bounds on minimal possible $\alpha$ for a given $\beta\in(0,1)$ in the asymptotics $n,k\to\infty$ and ${n\over
k}=\rho$.}\label{fig:eval}
\end{figure}

\subsection{Evaluation}

In this section we evaluate our bounds. We consider the asymptotic setting $k\to \infty$ and $n=\rho k$ where $\rho$ is
fixed. In Fig.~\ref{fig:eval1} ($\rho=3$) and Fig.~\ref{fig:eval2} ($\rho=1/3$) we plot the various 
bounds on the region of asymptotically feasible pairs $(\alpha,\beta)$:
\begin{itemize}
\item For $\rho=3$ the lower bound for $0<\alpha <1/2$ is~\eqref{eq:ccsam} from Theorem~\ref{th:main1}; for $1/2 \le \alpha \le 1$ is Theorem~\ref{th:main3}.
\item For $\rho=1/3$ the lower bound (for all $\alpha$) is~\eqref{eq:it}. In this case the other two
bounds,~\eqref{eq:ccsam} and~\eqref{eq:tm3}, are strictly worse.
\item For $\rho=3$ the straight dashed line denotes performance of the repetition map~\eqref{eq:rep}.
\item For $\rho={1\over3}$ the straight dashed line denotes performance of the \textit{majority-vote} map. Namely
$f:\FF_2^{3n}\to\FF_2^n$ gives a majority vote for every one of $3$-bit blocks. It is clear that for all $x,x'$ we have:
$$ |f(x) - f(x')| \le \beta n \quad\implies\quad |x-x'| \le {2+\beta\over 3} n $$
Indeed any pair of $3$-bit strings for which majority-vote agrees can be at most Hamming distance $2$ away (as $001$ and
$010$).%
\apxonly{In general, for odd $1\over \rho$ we have:
   $$ \alpha = \rho \beta + \bar \rho $$
}
\item Finally, the curved dashed line corresponds to the \textit{separation
map} defined as follows. Fix $\alpha \in (0,1)$ and cover $\FF_2^k$ with balls of radius $\alpha k/2$. It is sufficient
to have $2^{k (1-h(\alpha/2))+o(k)}$ such balls. Also consider a packing of balls of radius $\beta n/2$ in $\FF_2^n$.
By Gilbert-Varshamov bound we know that we can select at least $2^{n (1-h(\beta)) + o(n)}$ such balls. Thus whenever
$$ k (1-h(\alpha/2))+o(k) \le n (1-h(\beta)) + o(n) $$
we can construct the map $f:\FF_2^k\to\FF_2^n$ that maps every point inside an $\alpha k/2$ ball to a center of the
corresponding packing ball. Clearly, such map will be an $(\alpha,\beta)$ map. Thus, asymptotically all pairs of
$(\alpha,\beta)$ s.t. $\beta \le 1/2$ and
$$ 1-h(\alpha/2) \ge \rho (1-h(\beta)) $$
are achievable.
\end{itemize}

Note that as $\beta \upto {1\over2}$ the bound~\eqref{eq:ccsam} becomes:
\begin{equation}\label{eq:ccsam_x}
	\alpha \ge {1\over2} - \sqrt{\rho} \left({1\over 2} - \beta\right) + o(1-2\beta)\,,\qquad \beta\to {1\over2}\,. 
\end{equation}
This is a significant improvement over what the simple bound~\eqref{eq:ccb} yields:
$$ \alpha \ge {1\over2} - \sqrt{\rho\over 2\log_2 e} (1-2\beta) \log_2 {1\over 1-2\beta} + o(h((1/2-\beta)^2)), \qquad
\beta\to {1\over2}\,. $$
In particular,~\eqref{eq:ccsam_x} has finite slope at $\alpha=\beta={1\over2}$ and furthermore as $\rho\downto 1$ the
slope-discontinuity at $(1/2,1/2)$, see Fig.~\ref{fig:eval1}, disappears. This last effect is a consequence of the
Navon-Samorodnitsky~\cite{navon2005delsarte} part of the bound~\eqref{eq:rsam_def}.

\subsection{On list-decodable codes}
One of the more interesting conclusions that we can draw from our bounds is the following. It is well known that
there is only finitely many balls of radius $\ge n(1/4+\epsilon)$ that can be packed inside $\FF_2^n$ without
overlapping (Plotkin bound). However, if one allows these balls to cover each point with multiplicity at most 3 then it
is possible to pack exponentially many balls~\cite{blinovsky1986bounds}. Thus, if one is allowed to decode into lists of
size 3, it is possible to withstand adversarial noise of weight ${1/4 + \epsilon}$ while still having non-zero
communication rate.

By setting $\beta = 1/2 + 2\epsilon$ and applying Theorem~\ref{th:main3} we figure out, however, that no matter how the
balls are labeled by $k$-bit strings, at least one ball of radius $1/4+\epsilon$ will contain a pair of points whose
labels differ in at least $(1/2 + 2\epsilon)k$ positions. So although list-decoding allows one to overcome the $1/4$
barrier, there is no hope (in the worst case) to recover any information bits from the labels. This is only true 
beyond the radius $1/4$, since of course, below $1/4$ one can use codes with list-1. Loosely speaking, we have a ``phase-transition'' in the
communication problem at noise-level $1/4$.

\subsection{Linear programming bound}

It is possible to write a linear program for $\theta_S(\bar H(k,\alpha k) \ltimes \bar H(n,\beta n))$ similarly to the
standard Delsarte's program~\eqref{eq:ts_ham1}-\eqref{eq:ts_ham2}. To that end, for an arbitrary polynomial $f(x,y)$ of
degree at most $k$ in $x$ and at most $n$ in $y$ we can write it as
$$ f(x,y) = 2^{-n-k} \sum_{i=0}^k \sum_{j=0}^n \hat f(i,j) \krawtk_i(x) \krawt_j(y)\,,$$
where $\krawtk_i(x)$ and $\krawt_j(y)$ are Krawtchouk polynomials~\eqref{eq:krawt}. With this definition of the Fourier
transform $\hat f$ we have:
\begin{align}
	\lefteqn{\theta_S(\bar H(k, \alpha k) \ltimes \bar H(n, \beta n))}&\nonumber\\
		&= 
		\max \left\{ {\hat f(0,0)\over f(0,0)}: \hat f \ge 0, f(x,y) = 0 \forall (x,y) \in
		\mathcal{D}\setminus(0,0) \right\}\label{eq:tsh_ham1}\\
		     &= \min\left\{ 2^{k+n} {g(0,0)\over \hat g(0,0)}: \hat g \ge 0, g(x,y)\le 0, \forall (x,y) \in
		     \mathcal{D}^c \setminus (0,0) \right\} 				\label{eq:tsh_ham2}
\end{align}
where $f,g$ are bi-variate polynomials of degree at most $(k,n)$, and 
\begin{align} \mathcal{D} &\eqdef \left\{(x,y): x \in [0,k]\cap \mathbb{Z}, y\in [0,n]\cap \mathbb{Z}, (x=0, y\neq 0)\mbox{~or~}(x>\alpha
k, y \le \beta n)\right\}\\
   \mathcal{D}^c &\eqdef \left\{(x,y): x \in [0,k]\cap \mathbb{Z}, y\in [0,n]\cap \mathbb{Z}, (0<x\le \alpha
   k)\mbox{~or~}(x\neq 0, y>\beta n) \right\} 
\end{align}   

\apxonly{For computing $\theta_S(\bar H(k, \alpha k) \boxtimes H(n, \beta n))$ we get same program but with
\begin{align} \mathcal{D} &\eqdef \left\{(x,y): x \in [0,k]\cap \mathbb{Z}, y\in [0,n]\cap \mathbb{Z}, (x=0, y\le \beta
n)\mbox{~or~}(x>\alpha
k, y \le \beta n)\right\}\\
   \mathcal{D}^c &\eqdef \left\{(x,y): x \in [0,k]\cap \mathbb{Z}, y\in [0,n]\cap \mathbb{Z}, (0<x\le \alpha
   k)\mbox{~or~}(y>\beta n) \right\} 
\end{align}   
}

The bound used in Theorem~\ref{th:main2} states
$$  
\theta_S(\bar H(k, \alpha k) \ltimes \bar H(n, \beta n)) \le 2^{k + nR_{LP2}(\beta) -kR_{Sam}(\alpha) + o(n)+o(k)}
$$
This bound corresponds to the following choice of $g(x,y)$ in~\eqref{eq:tsh_ham2}:
$$ g(x,y) = f_1(x) g_1(y)\,,$$
where $f_1(x)$ is the Samorodnitsky assignment~\cite{AS01} in the primal for $\theta_S(H(k, \alpha
k))$ and $g_1(y)$ is a standard choice of McEliece-Rodemich-Rumsey-Welch~\cite{MRRW77} in the dual for
$\theta_S(H(n,\beta n))$. In fact, any primal $f_1$ and any dual $g_1$ give a candidate for $g(x,y)$. Thus, we have 
$$ \theta_S(\bar H(k, \alpha k)\ltimes \bar H(n,\beta n)) \le 2^{k+n} {f_1(0)\over \hat f_1(0)} {g_1(0)\over \hat g_1(0)}\,.$$
Optimizing over all $f_1$ and $g_1$ we get
$$ \theta_S(\bar H(k, \alpha k)\ltimes \bar H(n,\beta n)) \le 2^k {\theta_S(H(n,\beta n))\over \theta_S(H(k,\alpha k))}\,.$$
(This, of course, is exactly how the bound was obtained in Lemma~\ref{th:sprod}.)

We were not able to find any choice of $g(x,y)$ in the dual problem~\eqref{eq:tsh_ham2} that beats the product
$f(x)g(y)$. This seems to be the most natural direction for improvement. Another open problem is to find an upper bound
on $\bbalpha(\bar H(k, \alpha k) \ltimes \bar H(n,\beta n))$ that does not follow from an upper bound on $\bbalpha(\bar
H(k,\alpha k) \boxtimes \bar H(n, \beta n))$ or to prove that these are asymptotically equivalent. \apxonly{\textbf{TODO!}}

\subsection{On repetition \& majority-vote}

By looking at Fig.~\ref{fig:eval1} we can see that relaxation of the minimum-distance property, cf. Def.~\ref{def:ab},
that we consider in this paper allows one to have non-zero rate even at ``minimum distance'' $\beta > 1/2$. However, in
this case $\alpha\ge \beta$ (Theorem~\ref{th:main3}) and furthermore repetition map~\eqref{eq:rep} is optimal (provided
$n/k\in\mathbb{Z}$). This raises a number of questions:
\begin{itemize}
\item Can one show that any $(\alpha,\beta)$-map in high-$\beta$ regime is  structured similarly to a repetition map? 
\item For the case when $n/k \not \in\mathbb{Z}$ (e.g. $\rho=3/2$), how do we asymptotically achieve $\alpha \approx
\beta$? \apxonly{It is fairly easy to show that it is not possible to have $\alpha=\beta$ for finite $n,k$ in this
case.}
\item The corresponding situation with majority-vote maps is even worse, here we need $k\over n$ be an odd integer. What
is the counterpart for even $k\over n$?
\item Finally, how do we smoothly interpolate between the ``non-smooth'' separation construction (that is not
even injective) and the repetition map? 

\apxonly{Note however, that it is possible to have an infinite sequence of
$(\alpha,\beta)$-maps with $1>\alpha>\beta>1/2$ that also has non-vanishing relative minimum distance. Just consider a
sequence of systematic code of rate $1/2$ and copy the systematic part twice. This achieves $\beta = {\alpha
+2\delta}/3$ where $\delta = 0.11$ is the min dist.}
\end{itemize}

In fact the last question was our main practical motivation for looking into the concept of $(\alpha,\beta)$-maps. We
do not have any good candidates at this point.

\subsection{On linear codes}\label{sec:lincode}

A natural approach to construct good $(\alpha,\beta)$ maps is to restrict to linear maps $f:\FF_2^k\to\FF_2^n$. A linear
$f$ is $(\alpha,\beta)$ if 
$$ |x| > \alpha k  \quad \implies \quad |f(x)|>\beta n\,. $$
Instead of working with this condition, we get a more invariant notion by considering the graph of $f$ that is just a
linear subspace of $L \subset \FF_2^{k+n}$. Different conditions can be stated on $L$ that will ensure that $L$ defines an
$(\alpha,\beta)$-map, or that it gives an independent set in $\bar H(k,\alpha k) \ltimes \bar H(n,\beta n)$, or even an
independent set in $\bar H(k, \alpha k) \boxtimes H(n,\beta n)$.

We state these conditions for a general field $\FF$ and also in a geometric language of~\cite{TSN90}. The extension of
the concept of an $(\alpha,\beta)$-map, cf.
Definition~\ref{def:ab}, and Hamming graphs $H_\FF(n,d)$, cf.~\eqref{eq:hnd_def}, to arbitrary field $\FF$ is obvious.

Suppose that we
have (not necessarily distinct) points 
$$ u_1, \ldots, u_k, v_1,\ldots v_n \in \PP^{m-1} $$
where $\PP^{m-1}$ is a projective space of dimension $m-1$ over the field $\FF$. For every codimension 1 hyperplane $H$
define
$$ Z_v(H) \eqdef \#\{j: v_j\in H\}, \quad Z_u(H) \eqdef \#\{i: u_i \in H\} \,.$$

By writing these two collections of points in homogeneous coordinates we get a $m \times (k+n)$ matrix over $\FF$, whose
row-span gives the linear subspace $L\subset \FF^{k+n}$. We then have the following statements:
\begin{enumerate}
\item If points $\{u_i, v_j, i\in[k], j\in[n]\}$ are not contained in any (codimension 1) hyperplane $H\subset
\PP^{m-1}$ and satisfy 
$$ \forall H: Z_v(H) \ge (1-\beta) n \implies k>Z_u(H) \ge (1-\alpha) k$$
then
\begin{equation}\label{eq:ab_linx}
		\bbalpha(\bar H_\FF(k, \alpha k) \boxtimes H_\FF(n, \beta n)) \ge |\FF|^m\,.
\end{equation}	
\item If points $\{u_i, i\in[k]\}$ are not contained in any (codimension 1) hyperplane $H\subset \PP^{m-1}$ and
$$ \forall H: Z_v(H) \ge (1-\beta) n \implies Z_u(H) \ge (1-\alpha) k $$
then
	\begin{equation}\label{eq:ab_linx2}
		\bbalpha(\bar H_\FF(k, \alpha k) \ltimes \bar H_\FF(n, \beta n)) \ge |\FF|^m 
\end{equation}	
	(note that assumption implies $k\ge m$ here).
\item If in addition to previous assumption we also have $k=m$, i.e. points $\{u_i, i\in[k]\}$ span $\PP^{k-1}$, then
there exists a linear $(\alpha,\beta)$-map and thus
	$$ \bar H_\FF(k, \alpha k) \to \bar H_\FF(n, \beta n)\,. $$
\end{enumerate}

Note that if $\{u_i\}$ span $\PP^{k-1}$ and $\alpha=0$ then condition~\eqref{eq:ab_linx2} simply states that one cannot
include more than $(1-\beta)n$ points from $\{v_j, j\in[n]\}$ into any hyperplane -- i.e. a standard geometric condition
for $[n,k,d]_q$-systems, cf.~\cite{TSN90}. Similarly to how $[n,k,d]_q$ systems exactly correspond to
$[n,k,d]_q$ linear codes, existence of points $\{u_i,v_j\}$ satisfying~\eqref{eq:ab_linx} and assumptions in item 3 is
\textit{equivalent} to existence of an $\FF$-linear $(\alpha, \beta; k, n)$-map.

As an example, consider $\FF=\FF_2$. We will construct a linear map $\FF_2^3 \to \FF_2^4$ by selecting seven points on the
binary projective plane $\PP_2^2$: $u_1,u_2,u_3$ are any points spanning $\PP^2$, 
$$ v_1=u_1, v_2=u_2, v_3=u_3 $$
and finally put $v_4$ to be the only point not contained in any of the lines $(v_1,v_2), (v_1,v_3),(v_2,v_3)$. See Fig.~\ref{fig:fano2}
for an illustration. It is easy to see that condition~\eqref{eq:ab_linx2} holds with $\alpha={2\over3}$ and
$\beta={3\over4}$, thus
$$ \bar H(3, 2) \to \bar H(4,3)\,. $$
Computing~\eqref{eq:tsh_ham2} with smaller $\alpha$ and larger $\beta$ shows that the code of Fig.~\ref{fig:fano2} is optimal.

\begin{figure}[t]
\centering
\includegraphics[width=.4\textwidth]{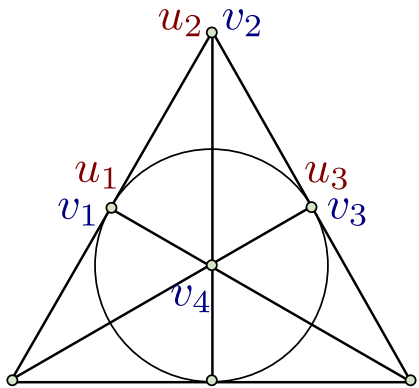}
\caption{Construction of the optimal $({2\over3}, {3\over4}; 3, 4)$-map by selecting seven points on the binary projective
plane.}\label{fig:fano2}
\end{figure}

\apxonly{
\section{Notes on linear constructions}
Equivalent definitions of $(\alpha,\beta)$ for linear $f(x)=xG$:
\begin{enumerate}
\item $|x|>\alpha k \implies |f(x)|>\beta n$ ~~(heavy goes to heavy)
\item $L$ is a subspace of $\FF_2^{k+n}$ with the property that $L\mapsto \mproj_{[k]} L$ is a bijection and
$$ \forall 0 \neq (x,y)\in L: |y| \le \beta n \implies |x| \le \alpha k $$
(codewords light in $n$-portion are also light in $k$-portion).
Equivalently,
$$ \forall (x,y)\in L: |y| > \beta n \mbox{~or~} |x| \le \alpha k\,. $$
\item There are $\{u_i\}_{i=1}^k$ and $\{v_j\}_{j=1}^n$ vectors in $\FF_2^k$ with the property that $\{u_i\}$ is a basis
and 
$$ \forall H\mbox{-hyperplane (codim 1) in~}\FF_2^k: |H\cap\{v_j\}| \ge \bar\beta n \implies |H\cap \{u_i\}| \ge \bar\alpha
k $$
That is the arrangement of $u$'s and $v$'s has the funny property that once any hyperplane contains many $v$'s it must
also contain many $u$'s.
\item Parity-check matrix point of view: $H=[C|D]$ is an $n\times (k+n)$ matrix of an $(\alpha,\beta)$-subspace $L$ iff 
$D$ is non-singular $n \times n$ and 
$$ \sum_{i\in S_2} d_i = \sum_{j\in S_1} c_j, |S_2| \le \beta n \quad \implies \quad |S_1| \le\alpha k\,,$$
where $d_i,c_j$ are columns of $C,D$. I.e. short-sums of $D$-columns either have no representation in $C$-columns or
only have short ones.
\item WLOG one may take $D=I_n$ above, then we are seeking $c_1,\ldots,c_k\in\FF_2^n$ with the property that
$$ \forall |S_1|>\alpha k: \quad \left|\sum_{i\in S_1} c_i\right| > \beta n\,. $$
I.e. when you add many $c_j$'s you necessarily get large weight.
\end{enumerate}

\bigskip
\textbf{Linear lower-bounds on $\bbalpha(X\ltimes Y)$ and $\bbalpha(X\boxtimes \bar Y)$.} Let $L$ be a
subspace of $\FF_2^{k+n}$, then:
\begin{enumerate} 
\item $L$ is an independent set of $\bar H(k, a) \ltimes \bar H(n, b)$
if 
$$ L\mapsto \mproj_1(L) \mbox{~is \textit{injective}} $$
and for every $(x,y) \in L$
$$ |x|> a \implies |y|>b  \qquad \mbox{or equiva:} \quad |x|\le a\mbox{~or~} |y|>b $$
(When $L\mapsto \mproj_1(L)$ is bijective, we also have a linear $(\alpha,\beta)$ map, and $|L|=|X|$.)
\item $L$ is an independent set of $\bar H(k,a)\boxtimes H(n,b)$ if for every $0\neq (x,y) \in L$
$$ |x|> a \mbox{~or~}x=0 \implies |y|>b \qquad \mbox{or equiva:} \quad 0<|x|\le a\mbox{~or~} |y|>b \,.$$
(so even injectivity of $\mproj_1$ is relaxed to allowing the fiber above $x=0$, i.e. $\mproj_1^{-1}(0)$, to be a
(linear) code in $\FF_2^n$ of minimum-distance bigger than $b$.)
\end{enumerate}
}

\apxonly{
\section{Random notes}
\begin{enumerate}
\item Questions around the top corner. $\bar H(k, k-1)$ is a disjoint union of $2^{k-1}$ cliques $K_2$. Thus
$$ \bar H(k, k-1) \to \bar H(n, n-1) $$
always exists, regardless of relations between $k$ and $n$. 

A similar question about Johnson graphs, which are just the Hamming spheres inside $\FF_2^k$, has a topological twist.
Indeed, since $\bar J(k, 2w-1,w)$ is just the Kneser graph $K(k,w)$ we have from Kneser conjecture:
 \begin{equation}\label{eq:topp1}
 	\bar J(k, 2w-1, w) \to \bar J(n, 2v-1, v)  \quad \implies k-2w \le n-2v \,. 
\end{equation} 
The bounds in this paper would be equivalent to comparing $\bbalpha(K(k,w))$ and $\bbalpha(K(n,v))$ (given by
Erd\"os-Ko-Rado) and result in:
 \begin{equation}\label{eq:topp2}
 	\bar J(k, 2w-1, w) \to \bar J(n, 2v-1, v)  \quad \implies {w\over k} \ge {v\over n}\,. 
\end{equation} 
Clearly,~\eqref{eq:topp1} provides extra restrictions in this case. For example, for the spheres of the same relative
radius, i.e. $w/k = v/n$, we get that $n\ge k$ is required. It would be very interesting to find similar
topological restrictions for the Hamming graphs.

\apxonly{For example, we could start with $\bar H(k, k-2)\to \bar H(n,n-2)$. Let us restrict to even $k,n$. Then
$\bbalpha$ and $\theta$ bounds seem to only imply $k \ge n$
(opposite to Johnson, why?).}

\item Another consideration about the top-corner comes from counting short odd-cycles: Let $\ell_1(G)$ be the length of
the shortest odd cycle. Then 
$$ G\to H \quad \implies \quad \ell_1(G) \ge \ell_1(H)\,.$$
The number of $m$-cycles from $0$ to $0$ in $\bar H(n, d)$ is given by
$$ \EE[(K_{d+1}^{(n)}(X_n) + \cdots K_{n}^{(n)}(X_n))^m], \qquad X_n \sim \mathrm{Bino}(n, 1/2). $$
\apxonly{So we get: $\bar H(k, \alpha k) \to \bar H(n, \beta n)$ implies that for all $m$:
\begin{align}
	& \EE[(K_{\beta n+1}^{(n)}(X_n) + \cdots K_{n}^{(n)}(X_n))^m] = 0 \quad \implies \quad\nonumber\\
	&\qquad \EE[(K_{\alpha k+1}^{(k)}(X_k) + \cdots K_{k}^{(k_k)}(X))^m] = 0 
			\label{eq:cycles}
\end{align}}
This can be used to show that $\ell_1(\bar H(n, n-2))=2\lceil {n+1\over 2} \rceil-1$,
$\ell_1(\bar H(n,n-3)))=2\lceil{n+1\over 4}\rceil+1$, etc.. In turn this gives simple examples of non-existence of
homomorphisms (not disproved by previous tools), e.g.:
$$  \bar H(2,0) \not\to \bar H(4,2) \not\to \bar H(6,4) \not\to \bar H(8,6) \not\to \cdots\,. $$
Can~\eqref{eq:cycles} lead to new asymptotic estimates?
\item Are there homomorphically equivalent ($G\to H$ and $H\to G$) pairs of complemented Hamming graphs?
\item For $\beta > 1/2$ the graphs $\bar H(n, \beta n)$ become very much locally disconnected, in the sense that they
stop containing small cliques and or any other graphs $G$ with large $\theta_S(\bar G)$. 

\item Nice question about LP: Frankl-Wilson (or Frankl-Rodl) show that any $S\subset \FF_2^k$ of size
$|S|>(2-\epsilon)^k$ must contain every distance $A_w(S) >0$ for all $w/n \in (\delta,1-\delta)$. Same for LP solutions?

\item Another question to discuss with Samorodnitsky: His assignment does not satisfy Chang's Lemma. Thus there cannot
be set $S$ s.t. its distance distribution equals Samorodnitsky. \textbf{Later:} Actually, it does satisfy Changs lemma.

\item \textbf{Spheres:} Note that we can prove also similar results for spheres: If 
	$$ kR_{GV}(\alpha) > n R_{LP}(\beta) $$
	then for any $f: \Sph^{k-1}\to\Sph^{n-1}$ there exist $x,x' \in \Sph^{k-1}$ s.t.
	$$ d(x,x') > \alpha, d(f(x), f(x')) \le \beta\,,$$
	where $d(\cdot, \cdot)$ is the geodesic distance (i.e. just angle between the vectors).

	\textbf{Question:} How does continuity of $f$ help? Maybe only under continuity we may guarantee the effective
	version such as $\forall S: \mathrm{vol}(S)\ge \ldots \exists x,x'\in S \ldots$?

\item Aaltonen has a method for producing bi-variate polynomials for the dual program. Check?

\item Note that the bounds in this paper correspond to considering
$$ K\rightarrow H_\alpha \to H_\beta \implies K\rightarrow H_\beta $$
(where $X\rightarrow Y$ is understood as $\theta_S(X\ltimes Y)=|X|$) and 
$$ H_\alpha \to H_\beta \to K \implies H_\alpha \to K $$
(this is because of~\eqref{eq:alphachi} and Kleitman). \textbf{Question:} Replace cliques $K$ with other graphs?!
\end{enumerate}
}

\section*{Acknowledgement}

This material is based upon work supported by the National Science Foundation under Grant No CCF-13-18620. Author is
grateful to the support of Simons Institute for the Theory of Computing (UC Berkeley), at which this work was finished.
Discussions with Prof. A. Mazumdar, A. Samorodnitsky and participants of the program on information theory at the Simons Institute were helpful.


\end{document}